\documentclass[reqno,10pt]{amsart}
\usepackage{amsmath,mathrsfs}
\usepackage{amssymb, amsmath}
\usepackage{graphicx}
\usepackage{pdfsync}
\usepackage{color}
\usepackage[colorlinks,
            linkcolor=red,
            anchorcolor=blue,
            citecolor=green
            ]{hyperref}
\usepackage{colonequals}
\newtheorem{definition}{Definition}[section]
\newtheorem{theorem}{Theorem}[section]
\newtheorem{lemma}{Lemma}[section]
\newtheorem{remark}{Remark}[section]
\newtheorem{corollary}{Corollary}[section]
\newtheorem{proposition}{Proposition}[section]

\numberwithin{equation}{section}

\newcommand{\beq}{\begin{equation}}
\newcommand{\eeq}{\end{equation}}
\newcommand{\ben}{\begin{eqnarray}}
\newcommand{\een}{\end{eqnarray}}
\newcommand{\beno}{\begin{eqnarray*}}
\newcommand{\eeno}{\end{eqnarray*}}
\let\f=\frac

\newcommand{\be}{\begin{equation} \label}
	\newcommand{\ee}{\end{equation}}
\newcommand{\bea}{\begin{eqnarray}\label}
	\newcommand{\eea}{\end{eqnarray}}
\newcommand{\bas}{\begin{eqnarray*}}
	\newcommand{\eas}{\end{eqnarray*}}
\newcommand{\bit}{\begin{itemize}}
	\newcommand{\eit}{\end{itemize}}

\newcommand{\N}{{\mathbb N}}

\newcommand{\R}{{\mathbb R}}

\newcommand{\pa}{\partial}

\newcommand{\ba}{\begin{aligned}}
\newcommand{\ea}{\end{aligned}}
 \def\na{\nabla}
 \newcommand{\lr}[1]{\langle #1 \rangle}

\begin{document}

\title[Non-cutoff Boltzmann equation near Traveling Maxwellians]{Global stability and scattering theory for non-cutoff Boltzmann equation with soft potentials in the whole space: weak collision regime}

\author{Ling-Bing He}
\address[Ling-Bing He]{Department of Mathematical Sciences, Tsinghua University, Beijng, 100084, P.R. China.}
\email{hlb@tsinghua.edu.cn}
\author{Wu-Wei Li}
\address[Wu-Wei Li]{Department of Mathematical Sciences, Tsinghua University, Beijng, 100084, P.R. China.}
\email{ww-li20@mails.tsinghua.edu.cn}

\begin{abstract}  A Traveling Maxwellian $\mathcal{M} = \mathcal{M}(t, x, v)$ represents a traveling wave solution to the Boltzmann equation   in the whole space $\R^3_x$(for the spatial variable). The primary objective of this study is to investigate the global-in-time stability of $\mathcal{M}$ and its associated scattering theory in   $L^1_{x,v}$ space for the non-cutoff Boltzmann equation with soft potentials when the dissipative effects induced by collisions are {\it weak}. We demonstrate the following results: (i) $\mathcal{M}$ exhibits Lyapunov stability; (ii) The perturbed solution, which is assumed to satisfy the same conservation law as $\mathcal{M}$, scatters in   $L^1_{x,v}$ space towards a particular traveling wave (with an explicit convergence rate), which may not necessarily be $\mathcal{M}$. The key elements in the proofs involve the formulation of the {\it Strichartz-Scaled Boltzmann equation}(achieved through the Strichartz-type scaling applied to the original equation) and the propagation of analytic smoothness.  
\end{abstract}


\maketitle
\setcounter{tocdepth}{1}
\tableofcontents

	
	\section{Introduction}
	In the present work, we will investigate the global-in-time stability  and the scattering theory for the non-cutoff Boltzmann equation with inverse power law potentials in the whole space $\R^3$  for the spatial variable. The classical  Boltzmann  equation  reads:
	\begin{equation}\label{Boltzmann}
		\partial_tF + v \cdot \nabla_xF = \mathsf{Q}(F, F),
	\end{equation}
	Here $F(t, x, v) \geq 0$ is a distributional function  of colliding particles which, at time $t>0$ and position $x \in \R^3$, move with velocity $v \in \R^3$. We emphasize that the   equation is one of the fundamental equations of mathematical physics and is a cornerstone of statistical physics.

	\subsection{Basic assumptions on the collision operator}
 The Boltzmann collision operator $\mathsf{Q}$ in \eqref{Boltzmann} acts only on the variable $v$, and is defined as
	\begin{equation*}
		\mathsf{Q}(F, G) = \int_{\mathbb{R}^3 \times \mathbb{S}^2}[F'_*G' - F_*G]B(v - v_*, \sigma)d\sigma dv_*.
	\end{equation*}
Some explanations are in order:

	\noindent(1). We use the shorthand $G = G(t, x, v), G' = G(t, x, v'), F_* = F(t, x, v_*), F'_* = F(t, x, v'_*)$ where $v', v'_*$ are defined in terms of $v, v_*, \sigma$ by
	\begin{equation*}
		v' = \frac{v + v_*}{2} + \frac{|v - v_*|}{2}\sigma, \:\:\: v'_* = \frac{v + v_*}{2} - \frac{|v - v_*|}{2}\sigma.
	\end{equation*}
	This follows the conservation law of momentum and energy of the collision:
	\begin{equation}\label{v'v*}
		v' + v'_* = v + v_*, \:\:\: |v'|^2 + |v'_*|^2 = |v|^2 + |v_*|^2.
	\end{equation}

	\noindent(2). The Boltzmann collision kernel $B = B(v - v_*, \sigma)$ is nonnegative and takes form of
	\begin{equation*}
		B(v - v_*, \sigma) = |v - v_*|^\gamma b(\cos\theta), \,\mbox{where}\,\cos\theta = \frac{v - v_*}{|v - v_*|} \cdot \sigma.
	\end{equation*}
	In the present, we will focus on the inverse power law potentials which read $\phi(r)=r^{-1/s}$ with $s \in (0, 1)$. By the physical argument, we will derive the explicit formula for $B$ which verifies 
	\begin{itemize}
		\item[$\mathbf{(A1)}$] The angular function $b(t)$ is not locally integrable and it satisfies
		\begin{equation*}
			\mathcal{K}\theta^{-1-2s} \le \sin\theta b(\cos\theta) \le \mathcal{K}^{-1}\theta^{-1-2s},~\mbox{with} ~0 < s < 1, ~\mathcal{K} > 0.
		\end{equation*}
		\item[$\mathbf{(A2)}$] The parameter $\gamma$ satisfies the condition $ \gamma \in (-3, 0)$.
		\item[$\mathbf{(A3)}$] Without lose of generality, we may assume that $B(v - v_*, \sigma)$ is supported in the set $0 \le \theta \le \pi/2$, i.e.$\frac{v - v_*}{|v - v_*|} \cdot \sigma \ge 0$, for otherwise $B$ can be replaced by its symmetrized form:
		\begin{equation*}
			\overline{B}(v - v_*, \sigma) = |v - v_*|^\gamma\big(b(\frac{v - v_*}{|v - v_*|} \cdot \sigma) + b(\frac{v - v_*}{|v - v_*|} \cdot (-\sigma))\big)\mathrm{1}_{\frac{v - v_*}{|v - v_*|} \cdot \sigma \ge 0},
		\end{equation*}
		where $\mathrm{1}_A$ is the characteristic function of the set $A$.
	\end{itemize}


	\subsection{Traveling Maxwellians, relative entropy and $H$-theorem}
	  {\it Traveling Maxwellians}, which must be the local Maxwellians, are referred to the traveling wave solutions to the Boltzmann equation. The precise definition is as follows(see also \cite{BGGL}):  
	\begin{definition}\label{deTM}
		$\mathcal{M} = \mathcal{M}(t, x, v) > 0$ is called a {\bf Traveling Maxwellian} if
		\begin{equation}\label{DefiTM}
			\partial_t\mathcal{M} + v \cdot \nabla_x\mathcal{M} = 0 = \mathsf{Q}(\mathcal{M}, \mathcal{M}).
		\end{equation}
	\end{definition}

 	 \begin{remark}
		 By the definition,   $\mathcal{M}$ is subject to two distinct mechanisms: the dispersion effect caused by the free transport operator $\partial_t + v \cdot \nabla_x$ and the dissipation effect induced by the collision operator $\mathsf{Q}$. These mechanisms impose contrasting influences on the behavior of $\mathcal{M}$. 
     \end{remark}

     \begin{remark} From the first equality in \eqref{DefiTM}, it is straightforward to deduce that $\mathcal{M}(t, x, v) = \mathcal{M}(0, x - tv, v)$. The second equality in \eqref{DefiTM} implies that $\mathcal{M}(t, x, v) = M_{[\rho, u, T]}$, where $M_{[\rho,u,T]}$ is the local Maxwellian defined as:
\begin{equation*}
			M_{[\rho,u,T]} := (2\pi T)^{-3/2}\rho\exp\big\{-\frac{|v - u|^2}{2T}\big\}.
		\end{equation*}
In the above expression, $\rho = \rho(t, x)$, $u = u(t, x)$, and $T = T(t, x)$ represent the macroscopic quantities, namely the density, velocity, and temperature, respectively. Therefore, a typical example of  {\it Traveling Maxwellians} would be given by:
\begin{equation*}
			\mathcal{M}(t, x, v) := e^{-|x - tv|^2 - |v|^2}.
		\end{equation*} 
	\end{remark}
	 
	  In \cite{LM}, the authors gave a complete classification of all  {\it Traveling Maxwellians}. We summarize as follows:
	\begin{proposition}\label{CharacterTM}
		$(a)$. The solution of Boltzmann equation $F = F(t, x, v)$ satisfies
		\begin{equation*}
			\frac{d}{dt}\int\phi(t, x, v)F(t, x, v)dxdv = 0.
		\end{equation*}
		if $\phi(t, x, v) = 1, v, x - tv, |v|^2, v \cdot (x - tv), |x - tv|^2, v \wedge (x - tv)$. Here $v \wedge x:=v x^T-xv^T$ is the skew tensor product.

		$(b)$. $\mathcal{M}$ is a Traveling Maxwellian if and only if 
		\begin{equation*}
			\ln\mathcal{M} \in \mathrm{span}\{1, v, x - tv, |v|^2, v \cdot (x - tv), |x - tv|^2, v \wedge (x - tv)\}. 
		\end{equation*} 

		$(c)$.  If $\int\mathcal{M}\, dvdx=1$ and $\int v\mathcal{M}\,dvdx=\int (x-tv)\mathcal{M}\,dvdx=0$, then there exist $(a,b,c)\in\R^3$ and  a $3 \times 3$ skew-symmetric matrix $A$ such that $a,c,ac-b^2>0$ and 
		\ben\label{ExformofM} \mathcal{M}(t, x, v) = \frac{\sqrt{\det Q}}{(2\pi)^3}\exp\bigg\{-\frac{1}{2}\big[a|v|^2 + 2bv \cdot (x - tv) + c|x - tv|^2 + 2v^\tau Ax\big]\bigg\},\een
		where $Q:=(ac-b^2)I+A^2$.
	\end{proposition}

	\begin{definition} Suppose that $F(t,x,v)$ is a global solution to \eqref{Boltzmann}.      
	Let $\Phi:F\rightarrow (m_0, u_0, y_0, a_0, b_0, c_0, A_0) \in \mathbb{R}^+ \times \mathbb{R}^3 \times \mathbb{R}^3 \times \mathbb{R}^+ \times \mathbb{R} \times \mathbb{R}^+ \times \mathbb{R}^{3\times3}$ given by
		\begin{equation}\label{DefMapPHI}
			\Phi(F)=(\Phi_1(F),\cdots,\Phi_7(F))^T:=\int\begin{bmatrix} 1 \\ v \\ x - tv \\ |v|^2 \\ v \cdot (x - tv) \\ |x - tv|^2 \\ v \wedge x \end{bmatrix}F(t, x, v)dxdv = \begin{bmatrix} m_0 \\ u_0 \\ y_0 \\ a_0 \\ b_0 \\ c_0 \\ A_0 \end{bmatrix}.
		\end{equation}
		 With reference to Proposition \ref{CharacterTM}, we define $\Phi$ as a conserved mapping. More specifically, $\Phi(F) = (\Phi_1(F), \cdots, \\ \Phi_7(F))^T$ corresponds to the conservation laws of mass, momentum, center of mass, energy, scalar momentum moment, scalar inertial moment, and angular momentum, respectively.
 \end{definition}

Next we show that   the conservation law \eqref{DefMapPHI} of $F$ will determine a unique {\it Traveling Maxwellian}. More precisely, we have the following proposition(see \cite{LM}):

 \begin{proposition}
\label{PhiFM}
		Let $F(t, x, v)$ be a nonnegative solution to \eqref{Boltzmann} with the initial data $F_0$ satisfying
		\begin{equation}\label{boundedF}
			\int(1 + |v|^2 + |x - tv|^2)F(t, x, v)dxdv=\int(1 + |v|^2 + |x |^2)F_0dxdv < \infty.
		\end{equation} Then there exists a unique {\it Traveling Maxwellian} $\mathcal{M}(t,x,v)$ satisfying that $\Phi(F)=\Phi(\mathcal{M})$.
\end{proposition}

 Thanks to Proposition \ref{PhiFM}, we can introduce the conception of relative entropy:   
\begin{definition} Let $F=F(t,x,v)$ and  $\mathcal{M}=\mathcal{M}(t,x,v)$ satisfy all the properties in Proposition \ref{PhiFM}. Then the relative entropy $\mathcal{H}[F|\mathcal{M}]$ and the associated  entropy dissipation are defined by
 \[\mathcal{H}[F|\mathcal{M}](t):= \int[F\ln\frac{F}{\mathcal{M}} - F + \mathcal{M}]dxdv.\]
  \[\mathcal{D}[F](t): = \int[F'_*F' - F_*F]\ln\frac{F'_*F'}{F_*F}Bd\sigma dvdv_*dx\ge0.\]
 The famous $H$-theorem can be described as follows: 
 \begin{equation}\label{HTheorem}
			\frac{d}{dt}\mathcal{H}[F|\mathcal{M}](t) + \mathcal{D}[F](t) = 0.
		\end{equation}
\end{definition}

Thanks to $H$-theorem, the relative entropy will be decreasing w.r.t. time variable. 

\subsection{Problem, difficulty and strategy} In this subsection, our primary objective is to illustrate the main purpose of our current work. We will begin by providing a brief review of previous research, after which we will delve into the primary challenges associated with the problems at hand and discuss the strategies employed to overcome them.

\subsubsection{Main problems and short review} In the present work, we are curious about two  problems on the non-cutoff Boltzmann equation \eqref{Boltzmann}:
\smallskip

(Q1). Does any {\it Traveling Maxwellian} $\mathcal{M}(t,x,v)$ exhibit the Lyapunov stability? What is the associated scattering theory in $L^1_{x,v}$ space? 

\smallskip

(Q2). If the perturbed solution $F$ satisfies that $\Phi(F)=\Phi(\mathcal{M})$, what is the dynamical behavior of the relative entropy $\mathcal{H}[F|\mathcal{M}](t)$? Does it converge to 0 as $t$ goes to infinity?

\smallskip

The Boltzmann equation in the whole(for the spatial variable) has been attracted lots of attention. In what follows, we give a short review of pertinent results  to $(Q1)$ and $(Q2)$.  The pioneering work is due to Kaniel-Shinbrot in \cite{KaSh}. They introduced the so-called Kaniel-Shinbrot iteration scheme to prove the local well-posedness of the cutoff equation. 

\noindent $\bullet$ In \cite{BGGL}, the question $(Q1)$ was thoroughly addressed for the cutoff Boltzmann equation with soft potentials under the condition that the mass of the  {\it Traveling Maxwellians} $\mathcal{M}=\mathcal{M}(t,x,v)$ is small, denoted as $\Phi_1(\mathcal{M})=m_0<1$. The authors also provided a negative response to $(Q2)$, demonstrating that the quantity  
\[\lim_{t\rightarrow\infty}\mathcal{H}[F|\mathcal{M}](t)\neq0.\]
For further insights into the stability of $\mathcal{M}$ without rotation (i.e., when $b=0$ and $A=0$ in \eqref{ExformofM}), as well as the long-term dynamics governed by dispersion, we recommend referring to \cite{Alonso-Gamba, Goudon, Lu, Toscani}.

\noindent $\bullet$ To ensure the stability of the vacuum in collisional kinetic models, Iller and Shinbrot established the first well-posedness of the cutoff equation globally in their work \cite{IlSh}. By leveraging the dispersion effect of the free transport equation, the authors in \cite{Arsenio,HJ2} demonstrated global existence in low regularity spaces using the Strichartz estimates. For a similar result employing the Wigner transform and bilinear spacetime estimates, we refer to \cite{CDP1}.
In a recent study by Luk \cite{Luk}, global stability of the vacuum was proven for the Landau equation with moderate soft potentials. Regarding the non-cutoff Boltzmann equation, Chaturvedi \cite{Ch} established similar results when $\gamma+2s\ge0$. Both authors effectively utilize the dispersive properties of the transport operator, employing vector field methods and space-time weights to demonstrate integrable time decay of the collisional operator.

\subsubsection{Difficulties and Strategies}\label{DerivSSBE} The primary challenge in addressing $(Q1)$ and $(Q2)$ is finding a balance between the dispersion effect caused by the free transport operator and the dissipation effect caused by the collision operator. To differentiate between these two distinct mechanisms, we propose the introduction of the {\it Strichartz-Scaled Boltzmann equation}. This equation is derived by applying the Strichartz-type scaling to the original equation and incorporating the Galilean invariance of the collision operator. Our overall strategy can be summarized as follows.
\smallskip

Let $F$ be a global solution to \eqref{Boltzmann} with the given initial data $F_0$. By Proposition \ref{PhiFM}, we assume that  $\mathcal{M}=\mathcal{M}(t,x,v)$ is the associated {\it Traveling Maxwellian} satisfying that $\Phi(F)=\Phi(M)$. In the subsequent steps, we will perform a series of variable transformations to convert $\mathcal{M}$ into the normalized {\it   Maxwellian} $\mathbf{M}=(2\pi)^{-3}e^{-\frac{1}{2}(|x|^2+|v|^2)}$. Additionally, we will carefully note the corresponding changes in variables and apply them to the equation.

\smallskip

\underline{\it Step 1: Normalizing the mass, the momentum and the center of mass}. Let $m:=\Phi_1(\mathcal{M})$,  $\mathcal{M}_1(t, x, v):= m^{-1}\mathcal{M}(t, x, v)$ and $F_1(t, x, v): = m^{-1}F(t, x, v)$. Now we have  $\Phi_1(\mathcal{M}_1)= 1$ and
	\begin{equation*}
		\partial_tF_1 + v \cdot \nabla_xF_1 = m\mathsf{Q}(F_1, F_1).
	\end{equation*}
Suppose that $u:=\Phi_2(\mathcal{M}_1)$, $y:=\Phi_3(\mathcal{M}_1)$ and let  
	  \beno \mathcal{M}_2(t, x, v) := \mathcal{M}_1(t, x + y + tu, v + u),\quad F_2(t, x, v) := F_1(t, x + y + tu, v + u).\eeno
	   These give that $\Phi_1(\mathcal{M}_2)=1$ and $\Phi_2(\mathcal{M}_2)=\Phi_3(\mathcal{M}_2)=0$. Moreover, it holds that
	   \ben 
	   \label{EqF2}&\partial_tF_2 + v \cdot \nabla_xF_2 = m\mathsf{Q}(F_2, F_2).\een

 \smallskip

\underline{\it Step 2: Scaling transform w.r.t. spatial and velocity variables}. Thanks to Proposition \ref{CharacterTM}(c),  there exist $(a,b,c)\in\R^3$ and  a $3 \times 3$ skew-symmetric matrix $A$ such that $a,c,ac-b^2>0$ and 
		\ben\label{ExformofM2} \mathcal{M}_2(t, x, v) = \frac{\sqrt{\det Q}}{(2\pi)^3}\exp\bigg\{-\frac{1}{2}\big[a|v|^2 + 2bv \cdot (x - tv) + c|x - tv|^2 + 2v^\tau Ax\big]\bigg\},\een
		where $Q:=(ac-b^2)I+A^2$.  We observe that 
	\begin{equation*}
		a|v|^2 + 2bv \cdot (x - tv) + c|x - tv|^2 + 2v^\tau Ax = \tau(t)^2\big|v + \frac{(b - ct + A)x}{\tau(t)^2}\big|^2 + c|x|^2 - \frac{|(b - ct + A)x|^2}{\tau(t)^2},
	\end{equation*}
	where $\tau(t) := \sqrt{a - 2bt + ct^2}$, and $(b - ct + A)x$ is short for $(b - ct)x + Ax$. Since 
 $ac - b^2 > 0$, we get that $a - 2bt + c t^2 > 0$ which makes sense the definition of $\tau(t)$. 

  Now we are in a position to introduce the scaling transform from $(t, x, v)$ to $(t, \tau(t)x, \tau(t)^{-1}(v - (b - ct + A)x))$. More precisely,  let
\ben\label{DefM3F3}
		&\qquad\mathcal{M}_3(t, x, v):= \mathcal{M}_2(t, \tau(t)x, \frac{v - (b - ct + A)x}{\tau(t)}),\quad
F_3(t, x, v):= F_2(t, \tau(t)x, \frac{v - (b - ct + A)x}{\tau(t)}).
\een
 
  We claim that $\mathcal{M}_3(t, x, v)= \frac{\sqrt{\det Q}}{(2\pi)^3}e^{-\frac{1}{2}(|v|^2 + (ac - b^2)|x|^2 - |Ax|^2)}.$ This easily follows the fact
	\begin{equation*}
		|v|^2 + c\tau(t)^2|x|^2 - |(b - ct + A)x|^2 = |v|^2 + (ac - b^2)|x|^2 - |Ax|^2
	\end{equation*}
	where $x \cdot Ax = 0$ since $A$ is skew-symmetry.  To derive the equation for $F_3$, by the definition, we first have
	\begin{equation*}
	\begin{aligned}
		\partial_tF_3 &= \partial_tF_2 + \frac{(ct - b)x}{\tau(t)} \cdot \nabla_xF_2 - \frac{(ct - b)v}{\tau(t)^3} \cdot \nabla_vF_2 + \frac{(ac - b^2)x + (ct - b)Ax}{\tau(t)^3} \cdot \nabla_vF_2 \\
		\nabla_xF_3 &= \tau(t)\nabla_xF_2 + \frac{ct - b + A}{\tau(t)}\nabla_vF_2, \:\:\:\:\:\: \nabla_vF_3 = \frac{1}{\tau(t)}\nabla_vF_2.
	\end{aligned}
	\end{equation*}
These immediately imply that
	\begin{equation*}
	\begin{aligned}
		\nabla_vF_2 &= \tau(t)\nabla_vF_3,  
		\nabla_xF_2 = \frac{1}{\tau(t)}\nabla_xF_3 - \frac{ct - b + A}{\tau(t)^2}\nabla_vF_2 = \frac{1}{\tau(t)}\nabla_xF_3 - \frac{ct - b + A}{\tau(t)}\nabla_vF_3 \\
		\partial_tF_2 &= \partial_tF_3 - \frac{(ct - b)x}{\tau(t)} \cdot \nabla_xF_2 + \frac{(ct - b)v}{\tau(t)^3} \cdot \nabla_vF_2 - \frac{(ac - b^2)x + (ct - b)Ax}{\tau(t)^3} \cdot \nabla_vF_2 \\
		&= \partial_tF_3 - \frac{(ct - b)x}{\tau(t)^2} \cdot \nabla_xF_3 + \frac{(ct - b)x \cdot (ct - b + A)}{\tau(t)^2}\nabla_vF_3 \\
		&+ \frac{(ct - b)v}{\tau(t)^2} \cdot \nabla_vF_3 - \frac{(ac - b^2)x + (ct - b)Ax}{\tau(t)^2} \cdot \nabla_vF_3.
	\end{aligned}
	\end{equation*}

	Now we can get the equation for $F_3$. On one hand,  we have
	\begin{equation*}
	\begin{aligned}
	&(\partial_tF_2 + v \cdot \nabla_xF_2)(t, \tau(t)x, \tau(t)^{-1}(v - (b - ct + A)x))
	\\&= \partial_tF_3 - \frac{(ct - b)x}{\tau(t)^2} \cdot \nabla_xF_3 + \frac{(ct - b)x \cdot (ct - b + A)}{\tau(t)^2}\nabla_vF_3 + \frac{(ct - b)v}{\tau(t)^2} \cdot \nabla_vF_3 \\
		&- \frac{(ac - b^2)x + (ct - b)Ax}{\tau(t)^2} \cdot \nabla_vF_3 + \frac{v - (b - ct + A)x}{\tau(t)} \cdot \frac{\nabla_xF_3 - (ct - b + A)\nabla_vF_3}{\tau(t)} \\
		&= \partial_tF_3 + \frac{(v - Ax) \cdot (\nabla_x - A\nabla_v)F_3 - (ac - b^2)x \cdot \nabla_vF_3}{\tau(t)^2}.
	\end{aligned}
	\end{equation*}
	On the other hand, since $\mathsf{Q}$ is Galilean invariant, one may have
 \begin{equation*}
		 m\mathsf{Q}(F_2, F_2)(t, \tau(t)x, \tau(t)^{-1}(v - (b - ct + A)x)) = m\tau(t)^{-3 - \gamma}\mathsf{Q}(F_3, F_3).
	\end{equation*}
	To see it, if $\bar{v} = \frac{v - (b - ct + A)x}{\tau(t)}$, we change of variables from $v_*$ to  $\bar{v}_*:= \frac{v_* - (b - ct + A)x}{\tau(t)}$. Correspondingly, we have 
$\bar{v}' = \frac{v' - (b - ct + A)x}{\tau(t)}, \bar{v}'_* = \frac{v'_* - (b - ct + A)x}{\tau(t)}.$
	This leads to the desired result. 
	
	Finally by \eqref{EqF2} we get that 
\ben\label{EquF3}
		\partial_tF_3 + \frac{(v - Ax) \cdot (\nabla_x - A\nabla_v)F_3 - (ac - b^2)x \cdot \nabla_vF_3}{\tau(t)^2} = m\tau(t)^{-3 - \gamma}\mathsf{Q}(F_3, F_3).
\een
\smallskip

\underline{\it Step 3: Linear transform w.r.t. spatial variable}.  Recalling  the definition of \eqref{DefM3F3}, we observe that 
	\begin{equation*}
		(ac - b^2)|x|^2 - |Ax|^2 = x^\tau((ac - b^2)I + A^2)x = x^\tau Qx
	\end{equation*}
	Since $Q$ is positive definite by Proposition \ref{CharacterTM}, there exists a positive definite matrix $B$ such that $Q = B^2$. Then $x^\tau Qx = x^\tau B^\tau Bx = |Bx|^2$ and $\sqrt{\det Q} = \det B$. We set
\ben\label{DefM4F4}
		\mathcal{M}_4(t, x, v):= (\det B)^{-1}\mathcal{M}_3(t, B^{-1}x, v),\quad F_4(t, x, v) = (\det B)^{-1}F_3(t, B^{-1}x, v).
	\een

  It is easy to verify that $\mathcal{M}_4=\mathbf{M}$.  To derive the equation for $F_4$, we compute that
	\beno
		&\partial_tF_4(t, x, v)  = (\det B)^{-1}\partial_tF_3(t, B^{-1}x, v),\,
		\nabla_xF_4(t, x, v) = (\det B)^{-1}D^{-1}\nabla_xF_3(t, B^{-1}x, v), \\
		&\nabla_vF_4(t, x, v)  = (\det B)^{-1}\nabla_vF_3(t, B^{-1}x, v),  
	\eeno
which implies that 
\beno 
		&(\partial_tF_3)(t, B^{-1}x, v) = (\det B)\partial_tF_4(t, x, v)F_4(t, x, v),\,
		(\nabla_xF_3)(t, B^{-1}x, v) = (\det B)D\nabla_xF_4(t, x, v), \\
		&(\nabla_vF_3)(t, B^{-1}x, v) = (\det B)\nabla_vF_4(t, x, v).
\eeno
	 
	 Thanks to \eqref{EquF3}, we  get the equation for $F_4$:
\beno
		&\qquad \partial_tF_4 + \frac{(v - AB^{-1}x) \cdot (B\nabla_x - A\nabla_v)F_4 - (ac - b^2)B^{-1}x \cdot \nabla_vF_4}{\tau(t)^2} = m(\det B)\tau(t)^{-3 - \gamma}\mathsf{Q}(F_4, F_4).
	\eeno
	 Using the facts $AB = BA$, $AB^{-1}x \cdot B\nabla_x = Ax \cdot \nabla_x$,  $v \cdot B\nabla_x = Bv \cdot\nabla_x, v \cdot A\nabla_v = A^\tau v \cdot \nabla_v = -Av \cdot \nabla_v$ and 
\beno
		AB^{-1}x \cdot A\nabla_v - (ac - b^2)B^{-1}x \cdot \nabla_v = x^\tau B^{-1}A^\tau A\nabla_v - (ac - b^2)x^\tau B^{-1}\nabla_v \\
		= -x^\tau B^{-1}(A^2 + (ac - b^2)I)\nabla_v = -x^\tau B\nabla_v = -Bx \cdot \nabla_v, 
\eeno
  we rewrite the equation of $F_4$ as follows
\ben\label{EqF4}
		 \partial_tF_4 + \frac{Bv \cdot\nabla_x - Bx \cdot \nabla_v + Av \cdot \nabla_v - Ax \cdot \nabla_x}{\tau(t)^2}F_4 =m(\det B)\tau(t)^{-3 - \gamma}\mathsf{Q}(F_4, F_4).
	\een 

If we keep track of all the tranformations, it is not difficult to see that:
\ben\label{connectionF4F} F_4(t,x,v)=m^{-1}(\det B)^{-1}F(t, \tau B^{-1}x + y + tu, \frac{v - (b - ct + A)B^{-1}x}{\tau} + u), \een where we recall that $(m,u,y)=(\Phi_1(F),m^{-1}\Phi_2(F),m^{-1}\Phi_3(F))$ and $(A,B,\tau,b,c)$ is determined by the coefficients appearing in \eqref{ExformofM2}.

	\subsubsection{Strichartz-Scaled Boltzmann Equation}    Now we are in a position to introduce the equation:
\begin{definition} Let $A,B\in  \R^{3 \times 3}$ and $a,b,c,m\in\R$ satisfy that
   (i). $m, a, c, ac-b^2 > 0$; (ii) $A$ is skew-symmetric and $B$ is positive definite satisfying that  $B^2 = (ac - b^2)I + A^2$. The {\bf Strichartz-Scaled Boltzmann Equation}  associated to the given {\it Traveling Maxwellian} $\mathcal{M}$ is defined by
 \ben\label{DefiSSBE}
 \partial_tG + \tau^{-2}\mathsf{T}G = m(\det B) \tau^{-3 - \gamma}\mathsf{Q}(G, G),
 \een
 where 
 \ben\label{DefiT}\mathsf{T}:= Bv \cdot\nabla_x - Bx \cdot \nabla_v + Av \cdot \nabla_v - Ax \cdot \nabla_x, \quad \tau: = \tau(t) = \sqrt{a - 2bt + ct^2}.\een 
Here $m=\Phi_1(\mathcal{M})$ and $(a,b,c,A,B)$ stems from $\mathcal{M}$  via \eqref{ExformofM2} and the argument in {\it Step 1} in Subsection \ref{DerivSSBE}. 
 \end{definition}

Next we show the basic properties of  {\it Strichartz-Scaled Boltzmann equation}  \eqref{DefiSSBE},  which are as the same as those of the original equation \eqref{Boltzmann}.
\smallskip

\noindent $\bullet$ {\it Conservation laws and $H$-theorem of {\it Strichartz-Scaled Boltzmann equation}  \eqref{DefiSSBE}}. We have 
 
\begin{proposition}\label{ConsHthSSBE} Let $G$ be a global  smooth solution to \eqref{DefiSSBE}.  
\begin{enumerate}
	\item[(i).]  If $\phi(t, x, v) = 1, v, x - tv, |v|^2, |x - tv|^2, v \cdot (x - tv), v \wedge x$, then
\ben \f{d}{dt}\int_{\R^6}\phi(t, \tau B^{-1}x, \frac{v - (b - ct + A)B^{-1}x}{\tau})G dxdv = 0;
\een
\item[(ii).]  If $\mathcal{H}[G|\mathbf{M}]:= \int_{\R^6}[G\ln\frac{G}{\mathbf{M}} - G + \mathbf{M}]dvdx$, then
\begin{equation}\label{HtheoremSSBE}
		\frac{d}{dt}\mathcal{H}[G|\mathbf{M}] + m(\det B)\tau^{-3 - \gamma}\mathcal{D}[G] = 0,
	\end{equation}
	where
	\ben\label{EnDissip}
	 \mathcal{D}[G] = \int_{\R^9 \times \mathbb{S}^2}[G'_*G' - G_*G]\ln\frac{G'_*G'}{G_*G}Bd\sigma dvdv_*dx \ge 0.
\een
\end{enumerate}
\end{proposition} 

\begin{proof} We only need to give a proof to result $(i)$ since result $(ii)$ is easily checked.  From \eqref{connectionF4F}, we  get that 
\ben\label{connectionGF2}G(t,x,v)=F_2(t, \tau B^{-1}x, \frac{v - (b - ct + A)B^{-1}x}{\tau}).\een Thus if   $\phi(t, x, v) = 1, v, x - tv, |v|^2, |x - tv|^2, v \cdot (x - tv), v \wedge x$, by change of variable, we have
\[\f{d}{dt}\int_{\R^6}\phi(t, \tau B^{-1}x, \frac{v - (b - ct + A)B^{-1}x}{\tau})G(t,x,v)dxdv =  \f{d}{dt}\int_{\R^6}\phi(t,x,v) F_2(t,x,v)dxdv=0,\]
since $F_2$ verifies \eqref{EqF2}. We end the proof. \end{proof}

\noindent $\bullet$ {\it Comments on {\it Strichartz-Scaled Boltzmann equation}  \eqref{DefiSSBE}}.  Several comments are in order:

\underline{(1).} Based on the reasoning presented in {\it Step 1} of Subsection \ref{DerivSSBE}, we observe that the original equation closely resembles the $F_2$-equation \eqref{EqF2} after normalizing the mass, the momentum and the center of mass of the given  {\it Traveling Maxwellian} $\mathcal{M}$. Therefore,
we can assume, without loss of generality, that the given {\it Traveling Maxwellian} $\mathcal{M}$ satisfies the conditions $\Phi_1(\mathcal{M})=1$ and $\Phi_2(\mathcal{M})=\Phi_3(\mathcal{M})=0$. In other words, we only need to consider the {\it Traveling Maxwellian} $\mathcal{M}$ in the form \eqref{ExformofM2}.
\smallskip

\underline{(2).} Let $G$ satisfy {\it Strichartz-Scaled Boltzmann equation}  \eqref{DefiSSBE} associate to  the {\it Traveling Maxwellian} $\mathcal{M}$ in the form \eqref{ExformofM2}. On one hand, by \eqref{connectionGF2}, we easily derive that 
\beno \|G(t)\|_{L^p_xL^q_v}\sim (t+1)^{3(\f1q-\f1p)}\|F_2(t)\|_{L^p_xL^q_v}. \eeno
On the other hand, we recall that the  Strichartz estimates for the free transport equation asserts that
\beno (t+1)^{3(\f1q-\f1p)}\|f(t)\|_{L^p_xL^q_v}\le\|f_0\|_{L^q_xL^p_v},  \eeno
where  $f(t,x,v)=f_0(x-tv,v)$.  
These observations indicate that the dispersion effect caused by the transport operator has been incorporated into \eqref{DefiSSBE} through the change of variables. This is the reason why we refer to it as the{\it Strichartz-Scaled Boltzmann equation}.
\smallskip

\underline{(3)}. In the context of the {\it Strichartz-Scaled Boltzmann equation} \eqref{DefiSSBE}, we can classify it into two cases based on the integrable time decay of the factor $\tau(t)^{-3 - \gamma}$ that appears on the right-hand side of the equation. This classification gives rise to the following definition:

\begin{definition} If $\tau(t)^{-3 - \gamma}$ is integrable (i.e., $\gamma > -2$), we refer to equation \eqref{DefiSSBE} as being in the {\it``weak collision regime''}. On the other hand, if $\tau(t)^{-3 - \gamma}$ is not integrable (i.e., $\gamma \in ]-3, -2]$), we classify equation \eqref{DefiSSBE} as being in the {\it ``strong collision regime''}.
\end{definition}

Obviously, in the  {\it ``strong collision regime''}, the collision effect prevails the dispersion effect by comparing the factors $\tau^{-3-\gamma}$ and $\tau^{-2}$ before the collision operator $\mathsf{Q}$ and the   transport operator $\mathsf{T}$. 
 Thus the relative entropy $\mathcal{H}[G|\mathbf{M}]$ and the $H$-theorem   will play the essential role in the long-time dynamics of the solution(see the progress in \cite{HJL}). 
 
 \smallskip
\underline{(4)}. Since vacuum is another stationary solutions to \eqref{DefiSSBE}, by \eqref{connectionGF2},  our strategy can also be applied to consider the global  stability of   vacuum state and the associated scattering theory in $L^p$(with $p>1$).


	\subsection{Notations and main results} 

	\subsubsection{Notations} We first list some notations used throughout the paper.

	\noindent$\bullet$ The bracket $\langle \cdot \rangle$ is defined by $\langle \cdot \rangle := \sqrt{1 + |\cdot|^2}$.  $1_A$ is the characteristic function of a set $A$.

	\noindent$\bullet$  $a \lesssim b$ is used  to indicate that there is a uniform constant $C$, which may be different on different lines, such that $a \le Cb$.  $a \gtrsim b$ means $b\lesssim a$.  If both $a\lesssim b$ and $b \lesssim a$, we write $a\sim b$.

	\noindent$\bullet$ For $f = f(x, v), g = g(x, v)$, notation $\langle f, g \rangle := \int f(x, v)g(x, v)dxdv$ is used to denote the inner product for $x, v$ variables. And $\langle \cdot, \cdot \rangle_v$ denotes the inner product for $v$ variable only.

	\noindent$\bullet$ If $A, B$ are two operators, then their commutator is defined by $[A, B] := AB - BA$.

	\noindent$\bullet$ The notation $\mathfrak{r} := \max\{\gamma + 2s, 0\}$ will be frequently used in the proof.

 \noindent$\bullet$ We denote $C(\lambda_1,\lambda_2,\cdots, \lambda_n)$ or $C_{\lambda_1,\lambda_2,\cdots, \lambda_n}$  by a constant depending on $\lambda_1,\lambda_2,\cdots, \lambda_n$.

\noindent  $\bullet$ For a multi-index
$\alpha =(\alpha_1,\alpha_2,\alpha_3) \in \mathbb{N}^{3}$, define
$|\alpha|:=    \alpha_1+\alpha_2+\alpha_3$. For  $\alpha \in \mathbb{N}^{3}$
denote $\partial^{\alpha}:=   \partial^{\alpha}_{x}$.

\noindent $\bullet$ For $l \geq 0$ and a function $f(v)$ on $\mathbb{R}^{3}$, define
\ben \label{l2-l-not-like-before}
|f|_{L^{2}_{l}}^{2} := |f\lr{\cdot}^l|_{L^{2}}^{2}, \quad |f|_{L^{2}}=|f|_{L^{2}_{0}}.
\een

\noindent $\bullet$
For $p,q\in[1,\infty], l\in\R$ and a function $f(x,v)$ on $\mathbb{R}^{3}\times \mathbb{R}^{3}$, define
\beno \|f\|_{L^p_xL^q_l}:=\bigg(\int_{\R^3} \big(\int_{\R^3} |f\lr{v}^l|^q dv\big)^{\f{p}{q}} dx\bigg)^{\f1{p}}.\eeno

\noindent $\bullet$
For $m \in \mathbb{N}, s,l\in \R$ and a function $f(x,v)$ on $\mathbb{R}^{3}\times \mathbb{R}^{3}$, define
\ben \label{not-mix-x-v-norm-energy}
\|f\|_{H^{m}_{x}H^{s}_{l}}^{2} :=      \sum_{|\alpha| \leq m}\|\pa^\alpha_x\lr{D_v}^sf\|_{L^2_xL^2_l}  , \quad \|f\|_{L^{2}_{x}L^{2}_{l}} :=    \|f\|_{H^{0}_{x}L^{2}_{l}}, \quad \|f\|_{H^{m}_{x}L^{2}}:=  \|f\|_{H^{m}_{x}L^{2}_{0}},
\een
where $\lr{D_v}^s f:=\int e^{2\pi i(v-u)\xi} \lr{\xi}^s f(u)dud\xi$.
\medskip

\subsubsection{Main results} Before stating our main results, we   address again several things:
\smallskip

\noindent $\bullet$ The questions $(Q1)$ and $(Q2)$ regarding the general  {\it Traveling Maxwellian} can be simplified to the form presented in \eqref{ExformofM2} using the reasoning outlined in {\it Step 1} of Subsection \ref{DerivSSBE}. Without loss of generality, from this point forward, we assume that the  {\it Traveling Maxwellian} $\mathcal{M}$ satisfies $\Phi_1(\mathcal{M})=1$ and $\Phi_2(\mathcal{M})=\Phi_3(\mathcal{M})=0$.

\noindent $\bullet$ Our strategy is based on the basic energy method, which is highly robust. Due to the presence of the integrable time decay factor $\tau(t)^{-2}$ before the transport operator $\mathsf{T}$ in \eqref{DefiSSBE}, we can simplify the form \eqref{ExformofM2} further by assuming $m=a=c=1$, $b=0$, $A=0$, and $B=I$. In other words, we only need to consider the typical  {\it Traveling Maxwellian} given by:
\[\mathcal{M}=(2\pi)^{-3}e^{-\frac{1}{2}(|x-tv|^2+|v|^2)}. \]
Under this assumption, the {\it Strichartz-Scaled Boltzmann equation} becomes:
\ben\label{DefiSSBES}
 \partial_tG + \lr{t}^{-2}\mathsf{T}G = \lr{t}^{-3 - \gamma}\mathsf{Q}(G, G),
 \een
 with
\ben\label{TtauS}\mathsf{T}=v\cdot\na_x-x\cdot\na_v, \quad \tau(t)=\lr{t}.\een
Furthermore, in this specific situation, the equation \eqref{EqF2} for $F_2$ coincides with the original equation \eqref{Boltzmann}. Thus, through \eqref{connectionGF2}, the solution $G$ to \eqref{DefiSSBES} and the solution $F$ to \eqref{Boltzmann} satisfy:
\ben\label{GtoF2S} G(t, x, v) =  F(t, \langle t \rangle x, \frac{tx + v}{\langle t \rangle}), \quad F(t,x,v)=G(t,\frac{x}{\langle t \rangle}, \langle t \rangle v - \frac{t}{\langle t \rangle}x). \een

	\smallskip

 \noindent$\bullet$ By utilizing \eqref{connectionGF2}, we can reduce the global stability analysis of the vacuum state to the original equation \eqref{Boltzmann} to the same problem for \eqref{DefiSSBE}. 
\smallskip

Before stating our main results, we introduce some additional notations:

 \noindent$\bullet$ Let $z := (x, v) \in \mathbb{R}^6,  \langle z \rangle := \sqrt{|x|^2 + |v|^2 + 1}$. It is easy to check that $\mathsf{T}\lr{z}=0$.

 \noindent$\bullet$ We introduce two vector fields $\mathsf{X}$ and $\mathsf{Y}$, which commute with $\mathsf{T}$(\eqref{TtauS}), \ben\label{DefiXY}
		\mathsf{X} = \frac{\partial_x - t\partial_v}{\langle t \rangle}, \:\:\: \mathsf{Y} = \frac{t\partial_x + \partial_v}{\langle t \rangle},\een 
  which implies that
	\begin{equation}\label{Upperxv}
		\partial_x = \frac{\mathsf{X} + t\mathsf{Y}}{\langle t \rangle}, \:\:\: \partial_v = \frac{\mathsf{Y} - t\mathsf{X}}{\langle t \rangle}.
	\end{equation}

 \noindent$\bullet$	Let $\mathsf{D}: = (\mathsf{X}, \mathsf{Y})$.  If $\mathsf{D}^\alpha:=X_1^{\alpha_1}X_2^{\alpha_2}X_3^{\alpha_3}Y_1^{\alpha_4}Y_2^{\alpha_5}Y_3^{\alpha_6}$ with $\alpha:=(\alpha_1,\cdots,\alpha_6)\in \N^6$, it holds that
	\begin{equation}\label{QDerivation}
		\mathsf{D}^\alpha\mathsf{Q}(F, G) = \sum_{\beta_1 + \beta_2 = \alpha}C_\alpha^{\beta_1}\mathsf{Q}(\mathsf{D}^{\beta_1}F, \mathsf{D}^{\beta_2}G).
	\end{equation}

 \noindent$\bullet$ Suppose $G$ and $F$ satisfy \eqref{GtoF2S}. Then \begin{equation}\label{FandG}
		\|\langle z \rangle^n\mathsf{X}^\alpha\mathsf{Y}^\beta G\|_{L^2_xL^2} = \|(|x - tv|^2 + |v|^2 + 1)^{n/2}\partial_x^\alpha(t\partial_x + \partial_v)^\beta F\|_{L^2_xL^2}.
	\end{equation}

\noindent$\bullet$	In what follows, we will use the following function spaces:

\underline{(1).}(Energy space) Let $m \in \mathbb{N}$ satisfy $(i)$. $(1 - s)m \ge 2, m \ge 4$; $(ii)$.  if $s < 1/2$, $(1-2s)m \ge 1$. We define:
\ben\label{FSGWP}
		|\!|\!|f(t)|\!|\!|_{N}^2 := \sum_{n = 1}^N\sum_{|\alpha| \le N - n}\|\langle z \rangle^{mn}\mathsf{D}^\alpha f(t)\|_{L^2_xL^2}^2.\een

\underline{(2).}(Analytic spaces) Let $ \delta\in]0, (1-3s)/(2s)], \:\:\: T_\gamma := \int_0^\infty\langle t \rangle^{-3 - \gamma}dt < +\infty$ with $s \in (0, 1/3),  \gamma \in (-2, -2s]$. 
 	We define two types of analytic norms:
\ben\label{FSAN}	\|f(t)\|_{AN}^2:=\sum_{\alpha\in\N^6}  \frac{q(t)^{2(|\alpha| + 1)}}{(|\alpha|!)^{2 + 2\delta}}\|\langle z \rangle^4\mathsf{D}^\alpha f(t)\|_{L^2_xL^2}^2,\\
\|f(t)\|_{MA}^2:=\sum_{\alpha\in\N^6} \frac{p(t)^{2(|\alpha| + 1)}}{(|\alpha|!)^{2 + 2\delta}}\|\langle z \rangle^4\mathsf{D}^\alpha f(t)\|_{L_x^2L^2}^2\label{FSMA}\een
 where \ben\label{Defiqt}
		  q(t) := (4T_\gamma)^{-1}(2T_\gamma - \int_0^t\langle \eta \rangle^{-3 - \gamma}d\eta) + \frac{1}{2} \in (\frac{3}{4}, 1];\\
		 p(t) := (4T_\gamma)^{-1}\int_0^t\langle \eta \rangle^{-3 - \gamma}d\eta + \frac{1}{4} \in [\frac{1}{4}, \frac{1}{2}).\label{Defipt}
	\een


 \smallskip
	Our first result is concerned about the global well-posedness and propagation of regularity of \eqref{DefiSSBE}.  
	\begin{theorem}[Strichartz-Scaled Boltzmann equation]\label{GWPPRSSBE} 
		Let $\gamma\in]-2,0[$ and consider the equation \eqref{DefiSSBES}  with the   initial data $G(0):=G_0\ge0$. 

	\underline{(i).(Global-wellposedness and Scattering theory)} Suppose that  $|\!|\!|(G-\mathbf{M})(0)|\!|\!|_{6}^2\le  \varepsilon_0$  with $\varepsilon_0\ll1$. Then \eqref{DefiSSBES} admits a unique and global solution $G=G(t,x,v)$satisfying that $G\ge0$ and \ben\label{EstGWP}\sup_{t\in[0,\infty[}|\!|\!|(G-\mathbf{M})(t)|\!|\!|_{6}^2\lesssim \varepsilon_0. \een
Moreover, $G\in C([0,\infty[;L^1_{x,v})$ and there exists a stationary function $G_\infty=G_\infty(x,v)$(see \eqref{DefiGinfty})  such that
\ben\label{scatteringL1} \|G(t)-G_\infty\|_{L^1_{x,v}}\lesssim \max\{\lr{t}^{-1}, \lr{t}^{-\gamma-2}\}. \een

 \underline{(ii).(Propagation of regularity)} Let $s \in (0, 1/3),  \gamma \in (-2, -2s]$. Suppose that   
 \[ 0<\varepsilon_2:=\|(G-\mathbf{M})(0)\|^2_{MA}\le \|(G-\mathbf{M})(0)\|^2_{AN}\le \varepsilon_1.\] Then  the solution  $G=G(t,x,v)$ to \eqref{DefiSSBES} satisfies that for $t>0$,
 \ben\label{EstPRLB}  \varepsilon_2\lesssim\|(G-\mathbf{M})(t)\|^2_{MA}\le \|(G-\mathbf{M})(t)\|^2_{AN}\lesssim \varepsilon_1\ll1.\een
 Moreover, $G_\infty\neq\mathbf{M}$.
	\end{theorem}
\begin{remark} We recall that the collision operator $\mathsf{Q}$, in the vicinity of $\mu=(2\pi)^{-3/2}e^{-\frac{1}{2}|v|^2}$, exhibits the following behavior: 
\beno
-\mathsf{Q}(\mu,g) \sim (-\triangle_v)^s(\langle v\rangle^\gamma g) + (-\triangle_{\mathbb{S}^2})^s(\langle v\rangle^\gamma g) + (v\cdot \nabla_v)(\langle v\rangle^\gamma g) + \text{L.O.T.},
\eeno
where $L.O.T.$ denotes lower-order terms. As a direct consequence, the anisotropic structure leads to a loss of weight in the $v$ variable, specifically the factor $\langle v\rangle^{\gamma+2s}$, when attempting to establish an upper bound for the collision operator in weighted Sobolev spaces. This motivates us to introduce the energy norm $|\!|\!|\cdot|\!|\!|_{N}$ to overcome this difficulty, particularly in the case where $\gamma+2s>0$. However, for $\gamma+2s\le 0$, we can simplify the energy norm $|\!|\!|\cdot|\!|\!|_{N}$ as follows:
	\[ \!|\!|f(t)|\!|\!\, _{N}^2 := \sum_{|\alpha|=0}^N \|\langle z \rangle^{4}\mathsf{D}^\alpha f(t)\|_{L^2_xL^2}^2,\] 
	which applies to the Cauchy problem of \eqref{DefiSSBES}. This reduction of the norm is primarily driven by the fact that we are interested in proving the propagation of smoothness using the norm $\|\cdot\|_{AN}$.
\end{remark}
Now we are in a position to recast the results in Theorem \ref{GWPPRSSBE} to the original equation \eqref{Boltzmann}.

\begin{theorem}[Boltzmann equation]\label{StabilityBE} Let $\mathcal{M}=(2\pi)^{-3}e^{-\frac{1}{2}(|x-tv|^2-|v|^2)}$ and $\gamma\in]-2,0[$.   Consider the equation \eqref{Boltzmann}  with the initial data $F(0):=F_0\ge0$. 

	\underline{(i).(Lyapunov stability and Scattering theory)} Suppose that $F_0$ satisfies that \[\sum_{n = 1}^6\sum_{|\alpha + \beta| \le 6 - n}\|(|x|^2 + |v|^2 + 1)^{mn/2}\partial_x^\alpha\partial_v^\beta(F-\mathcal{M})(0)\|_{L^2_xL^2}^2 \le \varepsilon_0\] with $\varepsilon_0\ll1$. Then \eqref{Boltzmann} admits a unique and global solution $F=F(t,x,v)$ satisfying that $F\ge0$ and  \[\sup_{t\in[0,\infty[}\bigg(\sum_{n = 1}^6\sum_{|\alpha + \beta| \le 6 - n}\|(|x-tv|^2 + |v|^2 + 1)^{mn/2}\partial_x^\alpha(t\pa_x+\partial_v)^\beta(F - \mathcal{M})(t)\|_{L^2_xL^2}^2\bigg) \lesssim \varepsilon_0.\] 
Moreover, there exists a stationary function $G_\infty=G_\infty(x,v)$ such that
\ben\label{scatteringL11} \int_{\R^6}|F(t)- G_\infty(\frac{x}{\langle t \rangle}, \langle t \rangle v - \frac{t}{\langle t \rangle}x)|dxdv\lesssim \max\{\lr{t}^{-1}, \lr{t}^{-\gamma-2}\}. \een

 \underline{(ii).($\mathcal{M}$ is {\it not} asymptotic stable)} Let $s \in (0, 1/3),  \gamma \in (-2, -2s]$. Suppose that  
 \[ 0<\varepsilon_2:=\|(F-\mathcal{M})(0)\|^2_{MA}\le \|(F-\mathcal{M})(0)\|^2_{AN}\le \varepsilon_1.\] Then there exists a universal constant $c=c(F_0)$ such that the solution  $F=F(t,x,v)$ to \eqref{Boltzmann} satisfies 
 \ben\label{FMdistance}  \inf_{t\in[0,\infty[}\|(F-\mathcal{M})(t)\|_{L^1_{x,v}}\ge c>0.\een
 which implies that for all $t>0$, $\mathcal{H}[F|\mathcal{M}](t)\ge c^2>0$.
\end{theorem}
 
\subsubsection{Comments on the theorems} As we explained at the beginning of this subsection, for the non-cutoff equation, Theorem \ref{GWPPRSSBE} and Theorem \ref{StabilityBE} demonstrate that all the Traveling Maxwellians $\mathcal{M}$, as defined in Definition \ref{deTM}, are Lyapunov stable. Furthermore, if $s \in ]0, \frac{1}{3}[$ and $\gamma\in]-2,2s]$, then all the Traveling Maxwellians $\mathcal{M}$ are not asymptotic  stable.
\smallskip

\underline{(1).} {\it Comment on the Lyapunov stability.} The analysis of Lyapunov stability for the {\it Traveling Maxwellians} $\mathcal{M}$ in relation to \eqref{Boltzmann} can be reduced to establishing the Lyapunov stability for the stationary solution $\mathbf{M}$ to \eqref{DefiSSBES}. In this regard, we can fully exploit the presence of the time decay factors $\lr{t}^{-2}$ and $\lr{t}^{-3-\gamma}$ appearing in \eqref{DefiSSBES}. Notably, our stability result does not rely on the dissipation property of the non-cutoff collision operator. As a result, our stability analysis can also be applied to prove $(Q1)$ for the cutoff equation, thereby removing the restriction $m<1$ imposed in \cite{BGGL}.
\smallskip

\underline{(2).} {\it Comment on the scattering theory.}  Our scattering theory in the $L^1_{x,v}$ space represents a novel development compared to the previous result. This advancement lies in our ability to clarify that the perturbed solution $F$ converges to another entity known as a {\it Traveling wave} $G_\infty(\frac{x}{\langle t \rangle}, \langle t \rangle v - \frac{t}{\langle t \rangle}x)$, accompanied by an explicit convergence rate.    Moreover, the traveling speed of $G_\infty$ is totally determined by $\Phi(F)$. To be more precise, if the solution $F$ is a perturbed solution of $\mathcal{M}$ satisfying that $\Phi(F)=\Phi(\mathcal{M})$, then by \eqref{connectionF4F}, \eqref{scatteringL1} turns to be  
\[\int_{\R^6}\bigg|F(t)-m(\det B)G_\infty\big(\frac{B(x - y - tu)}{\tau}, \tau(v - u)
 + \frac{(b - ct + A)(x - y - tu)}{\tau}\big)\bigg|dxdv\lesssim \max\{\lr{t}^{-1}, \lr{t}^{-\gamma-2}\}. \] Here $(m,u,y)=(\Phi_1(F),m^{-1}\Phi_2(F),m^{-1}\Phi_3(F))$ and $(A,B,\tau,b,c)$ is determined by $\Phi(F)$ through \eqref{ExformofM2}.

\underline{(3).} {\it Comment on the result that $\mathcal{M}$ is not asymptotic stable.} We have: 

$\bullet$ For the {\it Strichartz-Scaled Boltzmann equation}, we prove result $(ii)$ in Theorem \ref{GWPPRSSBE} by two steps. Firstly, we demonstrate the propagation of the analytic norm $\|\cdot\|_{AN}$. Subsequently, we utilize this result to establish the propagation of the lower bound of the analytic norm $\|\cdot\|_{MA}$. These steps involve the construction of appropriate functions $p(t)$ and $q(t)$, as defined in (\ref{Defipt}-\ref{Defiqt}). As a direct consequence, we conclude that $G_\infty$ does not equal   $\mathbf{M}$.
\smallskip

$\bullet$ By \eqref{GtoF2S}, result $(ii)$ in Theorem \ref{GWPPRSSBE} implies that for the original equation \eqref{Boltzmann},  $F$ consistently maintains a distance from $\mathcal{M}$ throughout the entire duration within the $L^1_{x,v}$ framework.  Therefore, we give a negative answer to $(Q_2)$. 
\smallskip

$\bullet$ In contrast to the findings in \cite{BGGL} regarding the cutoff equation, we are unable to establish a proof that any perturbed solution will maintain a distance from the corresponding {\it Traveling Maxwellian} $\mathcal{M}$  as initially prescribed. The primary reason for this limitation lies in the dissipation from the non-cutoff collision operator, which serves as the primary hurdle in demonstrating the desired outcome. This is the main motivation to consider the problem in the analytic function spaces defined in (\ref{FSAN}-\ref{FSMA}).



	\section{Proof of Theorem \ref{GWPPRSSBE}: global well-posedness and the scattering theory}  

  This section is devoted to the detailed proof of the global well-posedness and the scattering theory to the {\it Strichartz-Scaled Boltzmann equation}. We divide the proof into several parts. 

\subsection{Reformulation of equation \eqref{DefiSSBES}} We rewrite the equation \eqref{DefiSSBES} in the perturbation framework. To do it, we set $g := G - \mathbf{M}$. Then \eqref{DefiSSBES} turns to be
\ben\label{Eq-g}
		\partial_tg + \langle t \rangle^{-2}\mathsf{T}g = \langle t \rangle^{-3 - \gamma}[\mathsf{Q}(G, g) + \mathsf{Q}(g, \mathbf{M})].
	\een
	Since $\mathsf{X}$ and $\mathsf{Y}$ defined in \eqref{DefiXY} commute with $\partial_t + \langle t \rangle^{-2}\mathsf{T}$, we derive that $[\partial_t + \langle t \rangle^{-2}\mathsf{T}, \mathsf{D}]=0$. This implies that 
	\begin{equation*}
		\partial_t[\langle z \rangle^{mn}\mathsf{D}^\alpha g] + \langle t \rangle^{-2}\mathsf{T}[\langle z \rangle^{mn}\mathsf{D}^\alpha g] = \langle t \rangle^{-3 - \gamma}\langle z \rangle^{mn}\mathsf{D}^\alpha[\mathsf{Q}(G, g) + \mathsf{Q}(g, \mathbf{M})].
	\end{equation*}
	where  $\mathsf{T}\langle z \rangle = 0$ is used. By the standard energy method, we easily derive that 
	\begin{equation*}
		\frac{1}{2}\frac{d}{dt}\|\langle z \rangle^{mn}\mathsf{D}^\alpha g\|_{L^2}^2 = \langle t \rangle^{-3 - \gamma}\langle\mathsf{D}^\alpha[\mathsf{Q}(G, g) + \mathsf{Q}(g, \mathbf{M})], \langle z \rangle^{2mn}\mathsf{D}^\alpha g \rangle.
	\end{equation*}
	By \eqref{QDerivation} and \eqref{FSGWP}, we get that
	\begin{equation}\label{EnergyG}
		\frac{1}{2}\frac{d}{dt}|\!|\!|g|\!|\!|_N^2 = \langle t \rangle^{-3 - \gamma}\sum_{n = 1}^N\sum_{|\alpha| \le N - n}\sum_{\alpha_1 + \alpha_2 = \alpha}C_\alpha^{\alpha_1}\langle \mathsf{Q}(\mathsf{D}^{\alpha_1}G, \mathsf{D}^{\alpha_2}g) + \mathsf{Q}(\mathsf{D}^{\alpha_1}g, \mathsf{D}^{\alpha_2}\mathbf{M}), \langle z \rangle^{2mn}\mathsf{D}^\alpha g \rangle.
	\end{equation}
	Now we may reduce the  estimates   to the control of $I$ and $J$ defined as follows:
	\begin{equation*}
	\begin{aligned}
		I &:= \sum_{\alpha_1 + \alpha_2 = \alpha}\langle \mathsf{Q}(\mathsf{D}^{\alpha_1}G, \langle z \rangle^{mn}\mathsf{D}^{\alpha_2}g) + \mathsf{Q}(\mathsf{D}^{\alpha_1}g, \langle z \rangle^{mn}\mathsf{D}^{\alpha_2}\mathbf{M}), \langle z \rangle^{mn}\mathsf{D}^\alpha g \rangle;\\
		J &:= \sum_{\alpha_1 + \alpha_2 = \alpha}\langle \langle z \rangle^{mn}\mathsf{Q}(\mathsf{D}^{\alpha_1}G, \mathsf{D}^{\alpha_2}g) - \mathsf{Q}(\mathsf{D}^{\alpha_1}G, \langle z \rangle^{mn}\mathsf{D}^{\alpha_2}g), \langle z \rangle^{mn}\mathsf{D}^\alpha g \rangle \\
		&+ \sum_{\alpha_1 + \alpha_2 = \alpha}\langle \langle z \rangle^{mn}\mathsf{Q}(\mathsf{D}^{\alpha_1}g, \mathsf{D}^{\alpha_2}\mathbf{M}) - \mathsf{Q}(\mathsf{D}^{\alpha_1}g, \langle z \rangle^{mn}\mathsf{D}^{\alpha_2}\mathbf{M}), \langle z \rangle^{mn}\mathsf{D}^\alpha g \rangle.\\
	\end{aligned}
	\end{equation*}
	Going further, we may split $I$ and $J$ into several parts:

	\smallskip\noindent$\bullet$ For $I$, we have $I:= I_1 + I_2 + I_3 + I_4$, where
	\begin{equation*}
	\begin{gathered}
		I_1 = \langle \mathsf{Q}(G, \langle z \rangle^{mn}\mathsf{D}^\alpha g), \langle z \rangle^{mn}\mathsf{D}^\alpha g \rangle; \quad
		I_4 = \sum_{\alpha_1 + \alpha_2 = \alpha}\langle \mathsf{Q}(\mathsf{D}^{\alpha_1}g, \langle z \rangle^{mn}\mathsf{D}^{\alpha_2}\mathbf{M}), \langle z \rangle^{mn}\mathsf{D}^\alpha g \rangle;\\
	\end{gathered}
	\end{equation*}
	If $s \ge 1/2$,
	\begin{equation*}
	\begin{aligned}
		I_2 &= \sum_{\substack{\alpha_1 + \alpha_2 = \alpha \\ |\alpha_1| = 1}}\langle \mathsf{Q}(\mathsf{D}^{\alpha_1}G, \langle z \rangle^{mn}\mathsf{D}^{\alpha_2}g), \langle z \rangle^{mn}\mathsf{D}^\alpha g \rangle; \quad 
		I_3 &= \sum_{\substack{\alpha_1 + \alpha_2 = \alpha\\ |\alpha_1| \ge 2}}\langle \mathsf{Q}(\mathsf{D}^{\alpha_1}G, \langle z \rangle^{mn}\mathsf{D}^{\alpha_2}g), \langle z \rangle^{mn}\mathsf{D}^\alpha g \rangle.\\
	\end{aligned}
	\end{equation*}
	If $s < 1/2$,
	\begin{equation*}
		I_2 = 0; \:\:\: I_3 = \sum_{\substack{\alpha_1 + \alpha_2 = \alpha\\ |\alpha_1| \ge 1}}\langle \mathsf{Q}(\mathsf{D}^{\alpha_1}G, \langle z \rangle^{mn}\mathsf{D}^{\alpha_2}g), \langle z \rangle^{mn}\mathsf{D}^\alpha g \rangle.
	\end{equation*}

	\smallskip\noindent$\bullet$ For $J$, we have $J = J_1 + J_2 + J_3$, where
	\begin{equation*}
	\begin{aligned}
		J_1 &= \langle \langle z \rangle^{mn}\mathsf{Q}(G, \mathsf{D}^\alpha g) - \mathsf{Q}(G, \langle z \rangle^{mn}\mathsf{D}^\alpha g), \langle z \rangle^{mn}\mathsf{D}^\alpha g \rangle; \\
		J_2 &= \sum_{\substack{\alpha_1 + \alpha_2 = \alpha\\ |\alpha_1| \ge 1}}\langle \langle z \rangle^{mn}\mathsf{Q}(\mathsf{D}^{\alpha_1}G, \mathsf{D}^{\alpha_2}g) - \mathsf{Q}(\mathsf{D}^{\alpha_1}G, \langle z \rangle^{mn}\mathsf{D}^{\alpha_2}g), \langle z \rangle^{mn}\mathsf{D}^\alpha g \rangle; \\
		J_3 &= \sum_{\alpha_1 + \alpha_2 = \alpha}\langle \langle z \rangle^{mn}\mathsf{Q}(\mathsf{D}^{\alpha_1}g, \mathsf{D}^{\alpha_2}\mathbf{M}) - \mathsf{Q}(\mathsf{D}^{\alpha_1}g, \langle z \rangle^{mn}\mathsf{D}^{\alpha_2}\mathbf{M}), \langle z \rangle^{mn}\mathsf{D}^\alpha g \rangle. 
	\end{aligned}
	\end{equation*}

	In what follows, we will give the estimates term by term.

  
	\subsection{Estimate of $I_1$}   We divide the collision operator $\mathsf{Q}$ into $\mathsf{Q} = \mathsf{Q}_1 + \mathsf{Q}_2$, where
	\begin{equation*}
	\begin{aligned}
		\mathsf{Q}_1(F, G):= \int F'_*(G' - G)Bd\sigma dv_*; \quad
		\mathsf{Q}_2(F, G):= G\int (F'_* - F_*)Bd\sigma dv_*. 
	\end{aligned}
	\end{equation*}

	\begin{proposition}\label{Q2Estimate}
		For smooth functions $g = g(v), h = h(v)$ and $f = f(v)$, there holds
		\begin{equation*}
			|\langle \mathsf{Q}_2(g, h), f \rangle_v| \lesssim \int |v - v_*|^\gamma|g_*hf|dvdv_* \lesssim (|g|_{L^1} + |g|_{L^\infty})|h|_{L^2}|f|_{L^2}.
		\end{equation*}
	\end{proposition}

	\begin{proof}
		 Thanks to the Cancellation Lemma \ref{Cancellation}, it is easy to check that   
		\beno
			|\langle \mathsf{Q}_2(g, h), f \rangle_v| \le \int|v - v_*|^\gamma|g_*hf|dvdv_*.
		\eeno
		Then the desired result follows by splitting 
		 the integral region into $\{|v - v_*| > 1\}$ and $\{|v - v_*| \le 1\}$. 
	\end{proof}

	For $ g \ge 0 $, we introduce the dissipation norm stemming from the collision operator. We set
	\ben\label{DefDgg}
		D_g(f):= \frac{1}{2}\int g_*(f' - f)^2Bd\sigma dvdv_*, \:\:\: \mathfrak{D}_g(f):= \int D_g(f)dx.
	\een
	\begin{proposition}\label{coercivity}
		For smooth functions $g= g(v) \ge 0, f = f(v)$, we have
		\begin{equation*}
			\langle \mathsf{Q}(g, f), f \rangle_v + D_g(f) \lesssim (|g|_{L^1} + |g|_{L^\infty})|f|_{L^2}^2.
		\end{equation*}
	\end{proposition}

	\begin{proof}
		Using the pre-post change of variables, one may derive that
		\begin{equation}\label{coerandQ}
		\begin{aligned}
			\langle \mathsf{Q}(g, f), f \rangle_v & = -\frac{1}{2}\int g_*(f' - f)^2Bd\sigma dv_*dv + \frac{1}{2}\int g_*[(f')^2 - f^2]Bd\sigma dv_*dv.\\
		\end{aligned}
		\end{equation}
		Since $g \ge 0$,  we have $D_g(f)\ge0$. Moreover the Cancellation Lemma \ref{Cancellation} yields that
		\begin{equation*}
		\begin{aligned}
			\langle \mathsf{Q}(g, f), f \rangle_v + D_g(f) &\sim \int g(v_*)f(v)^2|v - v_*|^\gamma dvdv_*  
			&\lesssim (|g|_{L^1} + |g|_{L^\infty})|f|_{L^2}^2.
		\end{aligned}
		\end{equation*}
		This ends the proof.
	\end{proof}

	\begin{lemma}\label{ControlofI1} It holds that
		\begin{equation*}
			I_1 + \mathfrak{D}_G(\langle z \rangle^{mn}\mathsf{D}^\alpha g) \lesssim |\!|\!|G|\!|\!|_5|\!|\!|g|\!|\!|_N^2.
		\end{equation*}
	\end{lemma}

	\begin{proof}
		Thanks to  Proposition \ref{coercivity}, we get that
		\begin{equation*}
			I_1 + \mathfrak{D}_G(\langle z \rangle^{mn}\mathsf{D}^\alpha g) \lesssim \int(|G|_{L^1} + |G|_{L^\infty})|\langle z \rangle^{mn}\mathsf{D}^\alpha g|_{L^2}^2dx \lesssim (\|G\|_{L^\infty_x L^1} + \|G\|_{L^\infty_x L^\infty})\|\langle z \rangle^{mn}\mathsf{D}^\alpha g\|_{L_x^2L^2}^2.
		\end{equation*}
		Thanks to \eqref{Upperxv}, by Sobolev embedding Theorem, we have
		\ben\label{SobolevD}
			\|G\|_{L^\infty_x L^1} \lesssim \|G\|_{H^2_xL_2^2} \lesssim \sum_{|\alpha| \le 2}\|\mathsf{D}^\alpha(\langle z \rangle^2G)\|_{L_x^2L^2} \lesssim |\!|\!|G|\!|\!|_3; 
			\|G\|_{L^\infty_x L^\infty} \lesssim \|G\|_{H^2_xH^2} \lesssim \sum_{|\alpha| \le 4}\|\mathsf{D}^\alpha G\|_{L^2_xL^2} \lesssim |\!|\!|G|\!|\!|_5.\een
		From these together with the definition \eqref{FSGWP}, we get the desired result.
	\end{proof}

	\begin{corollary}\label{coerupper}
		From \eqref{coerandQ} and Corollary \ref{HLB1} we get that
		\begin{equation*}
			D_g(f) \lesssim |g|_{L_4^2}|f|_{H_{\mathfrak{r}/2}^s}^2 + (|g|_{L^1} + |g|_{L^\infty})|f|_{L^2}^2.
		\end{equation*}
	\end{corollary}


	\subsection{Estimate of $I_2$}
	Since $I_2 = 0$ when $s < 1/2$, we only need to consider the case $s \ge 1/2$. In this situaiton, we recall that $I_2$ is defined by
	\begin{equation*}
		I_2 = \sum_{\substack{\alpha_1 + \alpha_2 = \alpha\\ |\alpha_1| = 1}}\langle \mathsf{Q}(\mathsf{D}^{\alpha_1}G, \langle z \rangle^{mn}\mathsf{D}^{\alpha_2}g), \langle z \rangle^{mn}\mathsf{D}^\alpha g \rangle.
	\end{equation*}
Using the facts that $|\alpha_1| = 1$ and $\mathsf{D}^\alpha = \mathsf{D}^{\alpha_1}\mathsf{D}^{\alpha_2}$, we further have  $I_2=P + Q - R$, where
	\begin{equation*}
	\begin{aligned}
		P &:= \langle \mathsf{Q}_1(\mathsf{D}^{\alpha_1}G, \langle z \rangle^{mn}\mathsf{D}^{\alpha_2}g), \mathsf{D}^{\alpha_1}(\langle z \rangle^{mn}\mathsf{D}^{\alpha_2}g) \rangle;\quad 
		Q := \langle \mathsf{Q}_2(\mathsf{D}^{\alpha_1}G, \langle z \rangle^{mn}\mathsf{D}^{\alpha_2}g), \mathsf{D}^{\alpha_1}(\langle z \rangle^{mn}\mathsf{D}^{\alpha_2}g) \rangle;\\
		R &:= \langle \mathsf{Q}(\mathsf{D}^{\alpha_1}G, \langle z \rangle^{mn}\mathsf{D}^{\alpha_2}g), (\mathsf{D}^{\alpha_1}\langle z \rangle^{mn})\mathsf{D}^{\alpha_2}g \rangle.\\
	\end{aligned}
	\end{equation*}

	By Proposition \ref{Q2Estimate} and \eqref{SobolevD}, we derive the estimate for $Q$:
	\begin{equation*}
		|Q| \lesssim (\|\mathsf{D}^{\alpha_1}G\|_{L^\infty_x L^1} + \|\mathsf{D}^{\alpha_1}G\|_{L^\infty_x L^\infty})\|\langle z \rangle^{mn}\mathsf{D}^{\alpha_2}g\|_{L^2_xL^2}\|\mathsf{D}^{\alpha_1}(\langle z \rangle^{mn}\mathsf{D}^{\alpha_2}g)\|_{L^2_xL^2} \lesssim |\!|\!|G|\!|\!|_5|\!|\!|g|\!|\!|_N^2.
	\end{equation*}
	For $R$, by applying  Corollary \ref{HLB1}, we  first have
	\begin{equation*}
		|R| \lesssim \|\mathsf{D}^{\alpha_1}G\|_{L^\infty_x L_4^2}\|\langle z \rangle^{mn}\mathsf{D}^{\alpha_2}g\|_{L^2_xH_{\gamma/2 + s}^s}\|(\mathsf{D}^{\alpha_1}\langle z \rangle^{mn})\mathsf{D}^{\alpha_2}g\|_{L^2_xH_{\gamma/2 + s}^s}.
	\end{equation*}
	By the definition \eqref{FSGWP}, we have $(1 - s)m \ge 2$. Thus by   interpolation inequality
	\begin{equation}\label{TransDW}
		|f|_{H_{\gamma/2 + s}^s} \le |f|_{H_1^s} \lesssim |f|_{H^1}^s|f|_{L_{1/(1 - s)}^2}^{1 - s} \lesssim |f|_{H^1} + |f|_{L_m^2}.
	\end{equation}
	From this together with the fact that $|\mathsf{D}^{\alpha_1}\langle z \rangle^{mn}| \lesssim \langle z \rangle^{mn}$, we easily derive that   $|R| \lesssim |\!|\!|G|\!|\!|_4|\!|\!|g|\!|\!|_N^2$.  

	\begin{remark}  If
		  $(1 - s)m \ge 2$, similarly to \eqref{TransDW}, we have
		\begin{equation}\label{TransDW1}
			|f|_{H_2^s} \lesssim |f|_{H^1} + |f|_{L_m^2}.
		\end{equation}
		If $s < 1/2$ and $(1 - 2s)m \ge 1$, we have
		\begin{equation}\label{TransDW2}
			|f|_{H_{\gamma + 2s}^{2s}} \le |f|_{H_1^{2s}} \lesssim |f|_{H^1} + |f|_{L_m^2}.
		\end{equation}
	\end{remark}

Next, we will show that $|P| \lesssim |\!|\!|G|\!|\!|_5|\!|\!|g|\!|\!|_N^2$. To do it, we first prove 

	\begin{lemma}\label{itgbp}
		For smooth functions $g = g(v), f = f(v)$, it holds that
		\begin{equation*}
			\langle \mathsf{Q}_1(g, f), \partial f \rangle_v = \frac{1}{4}\int(\partial g)_*(f' - f)^2Bd\sigma dvdv_* + \frac{1}{2}\int(g'_* - g_*)(f' - f)\partial fBd\sigma dvdv_*.
		\end{equation*}
		for $\partial =\mathsf{X}, \mathsf{Y}$.
	\end{lemma}

	\begin{proof} We prove it in the spirit of \cite{Ch}. Because of \eqref{DefiXY}, we only need to prove the desired result in the case that $\partial = \partial_x, \partial_v$.

		\noindent$\bullet$ {\it Case 1: $\partial = \partial_x$.}  By the definition of $\mathsf{Q}_1$ we write
		\begin{equation*}
			\langle \mathsf{Q}_1(g, f), \partial f \rangle = \int g'_*(f' - f)\partial fBd\sigma dvdv_*.
		\end{equation*}
		For brevity, in what follows, we drop the integrals $\int\cdots d\sigma dvdv_*$ to get the equality. Applying integration by parts twice, we get that
		\begin{equation*}
		\begin{aligned}
			&g'_*(f' - f)\partial fB = -(\partial g)'_*(f' - f)fB - g'_*[(\partial f)' - \partial f]fB \\
			= &(\partial g)'_*(f' - f)^2B - (\partial g)'_*(f' - f)f'B - g'_*[(\partial f)' - \partial f]fB \\
			= &(\partial g)'_*(f' - f)^2B + g'_*[(\partial f)' - \partial f]f'B + g'_*(f' - f)(\partial f)'B - g'_*[(\partial f)' - \partial f]fB \\
			= &(\partial g)'_*(f' - f)^2B + g'_*[(\partial f)' - \partial f](f' - f)B + g'_*(f' - f)(\partial f)'B.\\
		\end{aligned}
		\end{equation*}
		From this, we further derive that
		\begin{equation*}
			2g'_*(f' - f)\partial fB = (\partial g)'_*(f' - f)^2B + 2g'_*(f' - f)(\partial f)'.
		\end{equation*}
		Using pre-post collision change of variables, we have
		\begin{equation}\label{prepostchange}
		\begin{aligned}
			(\partial g)'_*(f' - f)^2B &= (\partial g)_*(f' - f)^2B; \quad
			2g'_*(f' - f)(\partial f)'B &= -2g_*(f' - f)\partial fB,\\
		\end{aligned}
		\end{equation}
		which implies that
		\begin{equation*}
			4g'_*(f' - f)\partial fB = (\partial g)_*(f' - f)^2B + 2(g'_* - g_*)(f' - f)\partial f.
		\end{equation*}

		\noindent$\bullet$ {\it Case 2: $\partial = \partial_v$.} We set $\partial_* := \partial_{v_*}$. Since $(\partial + \partial_*)B = 0$ and $\partial_*f = 0$, by integration by parts twice, we derive that
		\begin{equation*}
		\begin{aligned}
			g'_*(f' - f)\partial fB &= -g'_*(f' - f)f\partial B - \partial[g'_*(f' - f)]fB  
			= g'_*(f' - f)f\partial_*B - \partial[g'_*(f' - f)]fB \\ 
			&= -\partial_*[g'_*(f' - f)]fB - \partial[g'_*(f' - f)]fB 
			= -(\partial + \partial_*)[g'_*(f' - f)]fB.\\
		\end{aligned}
		\end{equation*}
		Denote $\mathsf{d} = \partial + \partial_*, \mathsf{d}f = \partial f$, and then we have
		\begin{equation*}
		\begin{aligned}
			&g'_*(f' - f)\partial fB = -\mathsf{d}g'_*(f' - f)fB - g'_*(\mathsf{d}f' - \mathsf{d}f)fB  
			= \mathsf{d}g'_*(f' - f)^2B - \mathsf{d}g'_*(f' - f)f'B \\&- g'_*(\mathsf{d}f' - \partial f)fB 
			= \mathsf{d}g'_*(f' - f)^2B + g'_*(\mathsf{d}f' - \partial f)f'B + g'_*(f' - f)\mathsf{d}f'B - g'_*(\mathsf{d}f' - \partial f)fB \\
			&= \mathsf{d}g'_*(f' - f)^2B + g'_*[\mathsf{d}f' - \partial f](f' - f)B + g'_*(f' - f)\mathsf{d}f',\\
		\end{aligned}
		\end{equation*}
		which implies that
		\begin{equation*}
			2g'_*(f' - f)\partial fB = \mathsf{d}g'_*(f' - f)^2B + 2g'_*(f' - f)\mathsf{d}f'.
		\end{equation*}
		By the chain rule, we have
		\begin{equation*}
		\begin{gathered}
			\mathsf{d}f' = (\partial + \partial_*)f(\frac{v + v_*}{2} + \frac{|v - v_*|}{2}\sigma) = (\partial f)';\quad
			\mathsf{d}g'_* = (\partial + \partial_*)g(\frac{v + v_*}{2} - \frac{|v - v_*|}{2}\sigma) = (\partial g)'_*.\\
		\end{gathered}
		\end{equation*}
		Thus we have
		\begin{equation*}
			2g'_*(f' - f)\partial fB = (\partial g)'_*(f' - f)^2B + 2g'_*(f' - f)(\partial f)'.
		\end{equation*}
		From this together with \eqref{prepostchange}, we conclude the desired result.
	\end{proof}

	\begin{lemma}\label{partial}
		For smooth functions $g = g(x,v), f = f(x,v)$ and $\partial = \mathsf{D}^\alpha$ with $|\alpha| = 1$, it  holds that
		\beno
			 |\langle \mathsf{Q}_1(g, f), \partial f \rangle_v| &\lesssim& (|\partial g|_{L^1} + |\partial g|_{L^\infty})|f|_{L^2}^2 + |\partial g|_{L_4^2}|f|_{H_{\mathfrak{r}/2}^s}^2 + |g|_{L_\mathfrak{r}^1}|f|_{L_\mathfrak{r}^2}|\partial f|_{L^2}
				 + |\nabla g|_{L^1}|f|_{H^1}|\partial f|_{L^2} \\&&+ (|\langle \cdot \rangle\nabla g|_{L^1} + |\langle \cdot \rangle\nabla g|_{L^\infty})|f|_{L_1^2}^2 + |\langle \cdot \rangle\nabla g|_{L_4^2}|f|_{H_{\gamma/2 + s + 1}^s}^2 
				 + |\langle \cdot \rangle\nabla g|_{L^1}|\partial f|_{L^2}^2 \\&&+ |\nabla g|_{L_\mathfrak{r}^1}|f|_{L_\mathfrak{r}^2}|\partial f|_{L^2} + |\nabla^2g|_{L_\mathfrak{r}^1}|f|_{L_\mathfrak{r}^2}|\partial f|_{L^2}.
		\eeno
		Here we recall that  $\mathfrak{r} := \max\{\gamma + 2s, 0\}$.
	\end{lemma}

	\begin{proof} Thanks to Lemma \ref{itgbp}, we only need to estimate
		\begin{equation*}
			\mathbb{I}_1:=\big|\int(\partial g)_*(f' - f)^2Bd\sigma dvdv_*\big| \: \text{ and } \: \mathbb{I}_2:=\big|\int(g'_* - g_*)(f' - f)\partial fBd\sigma dvdv_*\big|.
		\end{equation*}
		\noindent  \underline{\it Estimate of $\mathbb{I}_1$.} Thanks to Corollary \ref{coerupper}, we have 
		\begin{equation*}
			\mathbb{I}_1 \le 2D_{|\partial g|}(f) \lesssim (|\partial g|_{L^1} + |\partial g|_{L^\infty})|f|_{L^2}^2 + |\partial g|_{L_4^2}|f|_{H_{\mathfrak{r}/2}^s}^2.
		\end{equation*}
		\noindent \underline{\it Estimate of $\mathbb{I}_2$.} We split $\mathbb{I}_2$ into two parts $\mathbb{I}_2^1$ and $\mathbb{I}_2^2$, which contain the non-singular region  $\{|v' - v| > 1\}$ and the singular region $\{|v' - v|\le 1\}$ respectively. Following the notations in Lemma \ref{nonsingular}, we have   
		\begin{equation}\label{nonsigularP}
			\mathbb{I}_2^1\lesssim |\mathcal{N}_1(g, f, \partial f)| + |\mathcal{N}_2(g, f, \partial f)| + |\mathcal{N}_2(g, \partial f, f)| + |\mathcal{N}_3(g, f, \partial f)|,
		\end{equation}
		which implies that    $\mathbb{I}_2^1\lesssim|g|_{L_\mathfrak{r}^1}|f|_{L_\mathfrak{r}^2}|\partial f|_{L^2}$.
		 
		For $\mathbb{I}_2^2$,  by Taylor's expansion that
    $g'_* - g_* = (\nabla g)_* \cdot (v'_* - v_*) + \int_0^1(1 - \kappa_1)(v'_* - v_*)^\tau\nabla^2g(\xi_1)(v'_* - v_*)d\kappa_1$with $\xi_1 = v_* + \kappa_1(v'_* - v_*)$, we have  $\mathbb{I}_2^2 \lesssim \sum_{i=1}^2 \mathbb{I}_2^{2,i}$, where
	\begin{equation*}
		\begin{aligned}
			 \mathbb{I}_2^{2,1} &= \int|(\nabla g)_*(f' - f)\partial f||v - v_*|^{\gamma + 1}\theta b(\cos\theta)1_{|v' - v| \le 1}d\sigma dvdv_*;\\
			 \mathbb{I}_2^{2,2} &= \int|\nabla^2g(\xi_1)||f\partial f||v - v_*|^{\gamma + 2}\theta^2b(\cos\theta)1_{|v' - v| \le 1}d\sigma dvdv_*d\kappa_1;\\
			 \mathbb{I}_2^{2,3} &= \int|\nabla^2g(\xi_1)||f'\partial f||v - v_*|^{\gamma + 2}\theta^2b(\cos\theta)1_{|v' - v| \le 1}d\sigma dvdv_*d\kappa_1.\\
		\end{aligned}
		\end{equation*}

		\noindent$\bullet$ {\it Estimate of $\mathbb{I}_2^{2,1}$}. We further split $\mathbb{I}_2^{2,1}$ into two parts $\mathbb{I}_2^{2,1,1}$ and $\mathbb{I}_2^{2,1,2}$, which contain the integral region $\{|v - v_*| \le 1\}$ and $\{|v - v_*| > 1\}$ respectively. For  $\mathbb{I}_2^{2,1,1}$,  using Taylor's expansion once again for $f$ to have
		 $f' - f = (v' - v) \cdot \int_0^1\nabla f(\xi)d\kappa$ with $\xi = v + \kappa(v' - v)$, we derive that
		 \beno \mathbb{I}_2^{2,1,1} \lesssim \int|(\nabla g)_*(\nabla f)(\xi)\partial f|\theta^2b(\cos\theta)d\sigma dvdv_*d\kappa, \eeno 
		 where  we use the  facts  that 
		 $|v - v_*|^{\gamma + 2}\mathrm{1}_{|v-v_*|\le1} \le 1$ and $|v'_* - v_*| = |v' - v|\sim |v - v_*|\theta$. By Cauchy-Schwartz inequality, we get that
		 \[ \mathbb{I}_2^{2,1,1}\lesssim \bigg(\int|(\nabla g)_*||\nabla f(\xi)|^2\theta^2b(\cos\theta)d\sigma dvdv_*d\kappa\bigg)^{\f12}\bigg(\int|(\nabla g)_*||\partial f|^2\theta^2b(\cos\theta)d\sigma dvdv_*d\kappa\bigg)^{\f12}.\]
		From the change of variables $v \to \xi$ and the fact
		\begin{equation}\label{xivchange}
	\det|\frac{d\xi}{dv}| = (1 - \frac{\kappa}{2})^3(1 + \frac{\kappa}{2 - \kappa}\cos\theta) \sim 1. \text{ for } \kappa \in [0, 1],  
		\end{equation}
		we conclude that $ \mathbb{I}_2^{2,1,1}\lesssim |\nabla g|_{L^1}|f|_{H^1}|\partial f|_{L^2}$. 

		To estimate $\mathbb{I}_2^{2,1,2}$, we use the fact  
		 $\langle v' \rangle(f' - f) = [(\langle v \rangle f)' - \langle v \rangle f] + (\langle v \rangle - \langle v' \rangle)f$
		  to yield that 
		 $ \mathbb{I}_2^{2,1,2} \le A_1+A_2$, where
		\begin{equation*}
		\begin{aligned}
			 A_1&:= \int|(\nabla g)_*||(\langle v \rangle f)' - \langle v \rangle f||\partial f|\theta\langle v' \rangle^{-1}|v - v_*|B1_{|v' - v| \le 1 < |v - v_*|}d\sigma dvdv_*;\\
			A_2 &:= \int|(\nabla g)_*||(\langle v \rangle - \langle v' \rangle)f||\partial f|\theta\langle v' \rangle^{-1}|v - v_*|^{\gamma + 1}b(\cos\theta)1_{|v' - v| \le 1 < |v - v_*|}d\sigma dvdv_*.\\
		\end{aligned}
		\end{equation*}
		Since  $\langle v' \rangle^{-1}|v - v_*| \lesssim \langle v_* \rangle$,   Cauchy-Schwartz inequality implies that
		\begin{equation*}
			A_1 \lesssim \int|\langle v_* \rangle(\nabla g)_*||(\langle v \rangle f)' - \langle v \rangle f||\partial f|\theta B1_{|v - v_*| > 1}d\sigma dvdv_* \le A_{1,1}^{1/2} \times A_{1, 2}^{1/2},
		\end{equation*}
		where $A_{1,1}:=\int|\langle v_* \rangle(\nabla g)_*||(\langle v \rangle f)' - \langle v \rangle f|^2Bd\sigma dvdv_*$ and $A_{1,2}:=\int|\langle v_* \rangle(\nabla g)_*||\partial f|^2\theta^2B1_{|v - v_*| > 1}d\sigma dvdv_*$.
		Due to Lemma \ref{coerupper}, we have
		\begin{equation*}
			A_{1, 1} = D_{|\langle \cdot \rangle\nabla g|}(\langle v \rangle f) \lesssim (|\langle \cdot \rangle\nabla g|_{L^1} + |\langle \cdot \rangle\nabla g|_{L^\infty})|f|_{L_1^2}^2 + |\langle \cdot \rangle\nabla g|_{L_4^2}|f|_{H_{\mathfrak{r}/2 + 1}^s}^2.
		\end{equation*}
		Since  $|v - v_*|^\gamma\mathrm{1}_{|v - v_*| > 1} \le 1$,    \eqref{kernel1} implies that $A_{1,2}\lesssim |\langle \cdot\rangle\nabla g|_{L^1}|\partial f|_{L^2}^2$. Then we conclude that  \[A_1\lesssim (|\langle \cdot \rangle\nabla g|_{L^1} + |\langle \cdot \rangle\nabla g|_{L^\infty})|f|_{L_1^2}^2 + |\langle \cdot \rangle\nabla g|_{L_4^2}|f|_{H_{\mathfrak{r}/2 + 1}^s}^2 + |\langle \cdot \rangle\nabla g|_{L^1}|\partial f|_{L^2}^2.\]

		For $A_2$, use the inequality $|\langle v' \rangle - \langle v \rangle| \lesssim |v' - v| \sim \theta|v - v_*|$, then  we have
		\begin{equation*}
			A_2 \lesssim \int|(\nabla g)_*||f\partial f|\theta^2|v - v_*|^{\gamma + 2}b(\cos\theta)1_{|v' - v| \le 1}d\sigma dvdv_*.
		\end{equation*}
		Following the notations in Lemma \ref{singular}, we have
		\begin{equation*}
			A_2 \lesssim \mathcal{J}_1(|\nabla g|, |f|, |\partial f|) \lesssim |\nabla g|_{L_\mathfrak{r}^1}|f|_{L_\mathfrak{r}^2}|\partial f|_{L^2}.
		\end{equation*}

		\noindent$\bullet$ {\it Estimate of $\mathbb{I}_2^{2,2}$ and $\mathbb{I}_2^{2,3}$}. We observe that
		\begin{equation}\label{xiv*change}
		\begin{aligned}
			\det|\frac{d\xi_1}{dv_*}| &= (1 - \frac{\kappa_1}{2})^3(1 + \frac{\kappa_1}{2 - \kappa_1}\cos\theta) \sim 1, \text{ for } \kappa_1 \in [0, 1], \theta \in [0, \frac{\pi}{2}].\\
		\end{aligned}
		\end{equation}
		If we denote $v_* = \psi(\xi_1)$, then $|v' - v| = \sin\frac{\theta}{2}|v - \psi(\xi)|$, and
		\begin{equation*}
		\begin{aligned}
			\mathbb{I}_2^{2,2} &= \int|\nabla^2g(\xi_1)||f\partial f||v - \psi(\xi_1)|^{\gamma + 2}\theta^2b(\cos\theta)1_{\sin\frac{\theta}{2} \le |v - \psi(\xi_1)|^{-1}}d\sigma dvd\xi_1d\kappa_1;\\
			\mathbb{I}_2^{2,3} &= \int|\nabla^2g(\xi_1)||f'\partial f||v - \psi(\xi_1)|^{\gamma + 2}\theta^2b(\cos\theta)1_{\sin\frac{\theta}{2} \le |v - \psi(\xi_1)|^{-1}}d\sigma dvd\xi_1d\kappa_1.\\
		\end{aligned}
		\end{equation*}
		Recall that $\xi_1 = v_* + \kappa_1(v'_* - v_*)$, and $|v'_* - v_*| = |v' - v| \le 1$. Then
		\begin{equation*}
			\langle v_* \rangle \lesssim \langle \xi_1 \rangle + \kappa_1\langle v'_* - v_* \rangle \lesssim \langle \xi_1 \rangle, \:\:\: \langle \xi_1 \rangle \lesssim \langle v_* \rangle + \kappa_1\langle v'_* - v_* \rangle \lesssim \langle v_* \rangle,
		\end{equation*}
		which yields that $\langle \psi(\xi_1) \rangle \sim \langle \xi_1 \rangle$. From this together with the fact that $|v - \psi(\xi_1)| = |v - v_*| \sim |v - \xi_1|$, we derive from Lemma \ref{singular}  that 
		\begin{equation*}
			\mathbb{I}_2^{2,2}+\mathbb{I}_2^{2,3} \lesssim |\nabla^2g|_{L_\mathfrak{r}^1}|f|_{L_\mathfrak{r}^2}|\partial f|_{L^2}.
		\end{equation*}

		 The desired result follows by putting together all the estimates. We end the proof.
	\end{proof}

 Now we are in a position to get the estimate of $I_2$.
	\begin{lemma}\label{ControlofI2}
		It holds that $|P|+I_2 \lesssim |\!|\!|G|\!|\!|_5|\!|\!|g|\!|\!|_N^2$.
	 \end{lemma}

	\begin{proof} Recall that $P = \langle \mathsf{Q}_1(\mathsf{D}^{\alpha_1}G, \langle z \rangle^{mn}\mathsf{D}^{\alpha_2}g), \mathsf{D}^{\alpha_1}(\langle z \rangle^{mn}\mathsf{D}^{\alpha_2}g) \rangle$ and  Lemma \ref{partial} can be rewritten as	\begin{equation*}
			|\langle \mathsf{Q}_1(g, f), \partial f \rangle_v| \lesssim (|\partial g|_{H^2} + |\partial g|_{L_4^2} + |g|_{H_1^2} + |g|_{L_5^2})|f|_{H_2^s}^2 + |g|_{H_4^2}(|f|_{H_2^s} + |f|_{H^1})|\partial f|_{L^2} + |g|_{H_3^1}|\partial f|_{L^2}^2.
		\end{equation*}
		The   desired result follows thanks to \eqref{TransDW1} and the definition \eqref{FSGWP}. 
	\end{proof}


	\subsection{Estimate of $I_3$ and $I_4$} We want to prove 
	\begin{lemma}\label{ControlofI34} If $N\ge6$, then $|I_3|+|I_4|\lesssim (|\!|\!|G|\!|\!|_N+1)|\!|\!|g|\!|\!|_N^2$.
\end{lemma}
\begin{proof} We split the proof into three cases.

	\underline{\it (1). Estimate of $I_3$ for $s < 1/2$.} In this situation, we recall that
	\begin{equation*}
		I_3 = \sum_{\substack{\alpha_1 + \alpha_2 = \alpha\\ |\alpha_1| > 0}}\underbrace{\langle \mathsf{Q}(\mathsf{D}^{\alpha_1}G, \langle z \rangle^{mn}\mathsf{D}^{\alpha_2}g), \langle z \rangle^{mn}\mathsf{D}^\alpha g \rangle}_{:=I_{3,\alpha_1,\alpha_2}}.
	\end{equation*}
	By Corollary \ref{HLB1}, we have
		$|I_{3,\alpha_1,\alpha_2}|\lesssim \int|\mathsf{D}^{\alpha_1}G|_{L_4^2}|\langle z \rangle^{mn}\mathsf{D}^{\alpha_2}g|_{H_{\gamma + 2s}^{2s}}|\langle z \rangle^{mn}\mathsf{D}^\alpha g|_{L^2}dx$.  
	Using \eqref{TransDW2} and imposing $L_x^\infty$ norm on $\mathsf{D}^{\alpha_1}G$ or $\langle z \rangle^{mn}\mathsf{D}^{\alpha_2}g$, we have
	\begin{equation*}
		|I_{3,\alpha_1,\alpha_2}| \lesssim |\!|\!|G|\!|\!|_{|\alpha_1| + 3}|\!|\!|g|\!|\!|_{n + |\alpha_2| + 1}|\!|\!|g|\!|\!|_N, \:\:\: |I_{3,\alpha_1,\alpha_2}| \lesssim |\!|\!|G|\!|\!|_{|\alpha_1| + 1}|\!|\!|g|\!|\!|_{n + |\alpha_2| + 3}|\!|\!|g|\!|\!|_N
	\end{equation*}
	By \eqref{FSGWP}, we have $n, |\alpha_1| \ge 1$ and $n + |\alpha_1| + |\alpha_2| \le N$. If $|\alpha_1| \le N - 3$,  the first estimate in the above implies that $|I_{3,\alpha_1,\alpha_2}| \lesssim |\!|\!|G|\!|\!|_N|\!|\!|g|\!|\!|_N^2$. If $|\alpha_1| \ge N - 2$, then $n + |\alpha_2| \le 2$ and the second estimate yields that $ |I_{3,\alpha_1,\alpha_2}|  \lesssim |\!|\!|G|\!|\!|_N|\!|\!|g|\!|\!|_N^2$ if $N\ge5$. We conclude that if $N\ge5$,
	\begin{equation*}
		|I_3| \lesssim |\!|\!|G|\!|\!|_N|\!|\!|g|\!|\!|_N^2.
	\end{equation*}

	\underline{\it (2). Estimate of $I_3$ for $s\ge1/2$.} In this case, we have $I_3=\sum\limits_{\substack{\alpha_1 + \alpha_2 = \alpha\\ |\alpha_1| \ge 2}}I_{3,\alpha_1,\alpha_2}$. Following the similar argument in the above, we have
	\begin{equation*}
		|I_{3,\alpha_1,\alpha_2}|\lesssim |\!|\!|G|\!|\!|_{|\alpha_1| + 3}|\!|\!|g|\!|\!|_{n + |\alpha_2| + 2}|\!|\!|g|\!|\!|_N, \:\:\: |I_{3,\alpha_1,\alpha_2}| \lesssim |\!|\!|G|\!|\!|_{|\alpha_1| + 1}|\!|\!|g|\!|\!|_{n + |\alpha_2| + 4}|\!|\!|g|\!|\!|_N,
	\end{equation*}
	which are enough to conclude that if $N\ge6$, then 
	\begin{equation*}
		|I_3| \lesssim |\!|\!|G|\!|\!|_N|\!|\!|g|\!|\!|_N^2.
	\end{equation*}
	 
\underline{\it (3). Estimate of $I_4$.} We recall that
	\begin{equation*}
		I_4 = \sum_{\alpha_1 + \alpha_2 = \alpha}\langle \mathsf{Q}(\mathsf{D}^{\alpha_1}g, \langle z \rangle^{mn}\mathsf{D}^{\alpha_2}\mathbf{M}), \langle z \rangle^{mn}\mathsf{D}^\alpha g \rangle.
	\end{equation*}
	By Corollary \ref{HLB1}, it is easy to check that $|I_4|\lesssim |\!|\!|g|\!|\!|_N^2$. This ends the proof. \end{proof}


	\subsection{Estimate of $J_1$} We begin with a lemma:
	\begin{lemma}\label{Comm1}
		For $\eta \in (0, 1), n\ge4$ and smooth functions $g= g(v) \ge 0, h = h(v)$ and $f = f(v)$, we have
		\beno
			\langle \langle z \rangle^n\mathsf{Q}(g, h) - \mathsf{Q}(g, \langle z \rangle^nh), f \rangle_v& \lesssim& (|g|_{L_4^2} + |g|_{H^2})|\langle z \rangle^nh|_{L^2}|f|_{L^2} 
				+ |g|_{L_n^2}|h|_{L_2^2}|f|_{L^2} \\&&+ \eta D_g(f) + \eta^{-1}|g|_{L_4^2}|\langle z \rangle^nh|_{L^2}^2.
	\eeno
	\end{lemma}

	\begin{proof} It is not difficult to compute that
		\begin{equation*}
			\mathfrak{I}:=\langle \langle z \rangle^n\mathsf{Q}(g, h) - \mathsf{Q}(g, \langle z \rangle^nh), f \rangle_v = \int g_*hf'(\langle z' \rangle^n - \langle z \rangle^n)Bd\sigma dvdv_*.
		\end{equation*}
		Let $\mathsf{W}_n(v):= \langle z \rangle^n$ and  $\chi=\chi(v,v_*,\sigma):=1_{|v - v_*| \le 4|v_*|}1_{|v' - v| > 1}$.  One has
			$\nabla \mathsf{W}_n(v) = nv\langle z \rangle^{n - 2}$ and $\nabla^2\mathsf{W}_n(v) = n\langle z \rangle^{n - 2}I + n(n - 2)v \otimes v\langle z \rangle^{n - 4}$.
		By Taylor's expansion
		\begin{equation*}
			\langle z' \rangle^n - \langle z \rangle^n = (v' - v) \cdot \nabla \mathsf{W}_n(v) + \frac{1}{2}(v' - v)^\tau\int_0^1\nabla^2\mathsf{W}_n(\xi)d\kappa (v' - v),
		\end{equation*}
		where $\xi = v + \kappa(v' - v)$, we decompose $\mathfrak{I}$ into four parts, i.e., $\mathfrak{I}=\sum_{i=1}^4\mathfrak{I}_i$, where  
		\begin{equation*}
		\begin{aligned}
			&\mathfrak{I}_1 = \int g_*hf(v' - v) \cdot \nabla \mathsf{W}_n(v)B(1-\chi)d\sigma dvdv_*; 
			\mathfrak{I}_2 = \int g_*h(f' - f)(v' - v) \cdot \nabla \mathsf{W}_n(v)B(1-\chi)d\sigma dvdv_*;\\
			 &\mathfrak{I}_3 = \f12\int g_*hf'(v' - v)^\tau\nabla^2\mathsf{W}_n(\xi) (v' - v)B(1-\chi)d\kappa d\sigma dvdv_*; 
			\mathfrak{I}_4 = \int g_*hf'(\langle z' \rangle^n - \langle z \rangle^n)B\chi d\sigma dvdv_*.\\
		\end{aligned}
		\end{equation*}

		\noindent\underline{\it Estimate of $\mathfrak{I}_1$}.	By Lemma \ref{lemsymetric} and the fact that $\nabla \mathsf{W}_n(v) \lesssim \langle z \rangle^{n - 1}$, we first get that
		\begin{equation*}
		\begin{aligned}
			|\mathfrak{I}_1| 
			&\lesssim \int|g_*\langle z \rangle^{n - 1}hf||v - v_*|^{\gamma + 1}\theta^2b(\cos\theta)(1-\chi)d\sigma dvdv_*.\\
		\end{aligned}
		\end{equation*} 
		Since $1-\chi \le 1_{|v - v_*| > 4|v_*|} + 1_{|v' - v| \le 1}$, we easily obtain that
        \[ |\mathfrak{I}_1|\lesssim (|g|_{L_3^2} + |g|_{H^2})|\langle z \rangle^nh|_{L^2}|f|_{L^2}. \]

		\noindent\underline{\it Estimate of $\mathfrak{I}_2$.} Using the condition that $g \ge 0$ and the notations in  Lemma \ref{singular},     we get from Cauchy-Schwartz inequality that 
		\beno
			&&|\mathfrak{I}_2|   
			\lesssim D_g(f)^{\f12}\bigg(\int g_*\langle z \rangle^{2n - 2}h^2|v - v_*|^{\gamma + 2}\theta^2b(\cos\theta)(1-\chi)d\sigma dvdv_*\bigg)^{\f12}\\&&\lesssim D_g(f)^{\f12}\bigg( \int g_*\langle z \rangle^{2n}h^2dvdv_*+\mathcal{J}_1(g, \langle z \rangle^{n - 1}h, \langle z \rangle^{n - 1}h)\bigg)\lesssim   \eta D_g(f) + \eta^{-1}|g|_{L_4^2}|\langle z \rangle^nh|_{L^2}^2,
		\eeno
        where we use again the fact  $1-\chi \le 1_{|v - v_*| > 4|v_*|} + 1_{|v' - v| \le 1}$.

		\noindent\underline{\it Estimate of $\mathfrak{I}_3$.} It is easy to verify that  $\langle \xi \rangle \lesssim \langle v \rangle$ on the support of $\chi$, which implies that $|\nabla^2\mathsf{W}_n(\xi)|\lesssim \langle z \rangle^{n - 2}$. Following the notations in Lemma \ref{kernel} and Lemma \ref{singular}, we have
		\begin{equation*}
			 |\mathfrak{I}_3|\lesssim \int|g_*\langle z \rangle^{n - 2}hf'||v - v_*|^{\gamma + 2}\theta^2b(\cos\theta)(1-\chi)d\sigma dvdv_*\le \mathcal{I}_2(|g|, \langle z \rangle^n|h|, |f|)+\mathcal{J}_2(|g|, \langle z \rangle^{n - 2}h, f),
		\end{equation*}
		 which yields that $|\mathfrak{I}_3| \lesssim |g|_{L_4^2}|\langle z \rangle^n h|_{L^2}|f|_{L^2}.$

		\noindent\underline{\it Estimate of $\mathfrak{I}_4$.} By the fact that $\langle z' \rangle^n \lesssim \langle z \rangle^n + |v' - v|^n$ and $n\ge4$, we   derive that 
		\beno  |\mathfrak{I}_4|\lesssim \int |g_*hf'|( \langle z \rangle^n+|v'-v|^n)B\chi d\sigma dvdv_*\lesssim \mathcal{P}(\langle \cdot \rangle^\mathfrak{r}|g|, \langle z \rangle^n|h|, |f|)+\int|g_*hf'||v - v_*|^{\gamma + n}\theta^4 b(\cos\theta)\chi d\sigma dvdv_*,  \eeno
        where we follow the notations  in Lemma \ref{nonsingular+}. Let us focus on the second term in the right-hand side.  Cauchy-Schwartz inequality yields that
        \beno &&\int|g_*hf'||v - v_*|^{\gamma + n}\theta^4 b(\cos\theta)\chi d\sigma dvdv_*\lesssim \bigg(\int|\langle v_* \rangle^ng_*|^2|h|\theta^2b(\cos\theta)d\sigma dvdv_*\bigg)^{1/2} 
			  \\
			  &&\times\bigg(\int|h||f'|^2\theta^5b(\cos\theta)d\sigma dvdv_*\bigg)^{1/2}\lesssim  |g|_{L_n^2}^2|h|_{L^1}|f|_{L^2}, \eeno 
        where   Lemma \ref{chv} and \eqref{kernel1} are employed. It is enough to conclude that   
		\begin{equation*}
			|\mathfrak{I}_4|\lesssim |g|_{L_4^2}|\langle z \rangle^nh|_{L^2}|f|_{L^2} + |g|_{L_n^2}|h|_{L_2^2}|f|_{L^2}.
		\end{equation*}

		Putting together  all the estimates together, we complete the proof of the lemma.
	\end{proof}

	\begin{lemma}\label{ControlofJ1} If $N\ge4$, we have 
		\begin{equation*}
			J_1 \lesssim \eta\mathfrak{D}_G(\langle z \rangle^{mn}\mathsf{D}^\alpha g) + \eta^{-1}|\!|\!|G|\!|\!|_N|\!|\!|g|\!|\!|_N^2.
		\end{equation*}
	\end{lemma}

	\begin{proof}
		Thanks to  Lemma \ref{Comm1}, we directly have
		\begin{equation*}
		\begin{aligned}
			|J_1| &= \big|\langle \langle z \rangle^{mn}\mathsf{Q}(G, \mathsf{D}^\alpha g) - \mathsf{Q}(G, \langle z \rangle^{mn}\mathsf{D}^\alpha g), \langle z \rangle^{mn}\mathsf{D}^\alpha g \rangle\big| \\
			&\lesssim \eta^{-1}(\|G\|_{L^\infty_x L_4^2} + \|G\|_{L^\infty_x H^2}\|\langle z \rangle^{mn}\mathsf{D}^\alpha g\|_{L^2_xL^2}^2 + \int|G|_{L_{mn}^2}|\mathsf{D}^\alpha g|_{L_2^2}|\langle z \rangle^{mn}\mathsf{D}^\alpha g|_{L^2}dx + \eta\mathfrak{D}_G(\langle z \rangle^{mn}\mathsf{D}^\alpha g) \\
			&\lesssim \eta\mathfrak{D}_G(\langle z \rangle^{mn}\mathsf{D}^\alpha g) + \eta^{-1}|\!|\!|G|\!|\!|_4|\!|\!|g|\!|\!|_N^2 + \int|G|_{L_{mn}^2}|\mathsf{D}^\alpha g|_{L_2^2}|\langle z \rangle^{mn}\mathsf{D}^\alpha g|_{L^2}dx.\\
		\end{aligned}
		\end{equation*}
		 We observe that if $n \le N - 2$,   \[\int|G|_{L_{mn}^2}|\mathsf{D}^\alpha g|_{L_2^2}|\langle z \rangle^{mn}\mathsf{D}^\alpha g|_{L^2}dx\le \|G\|_{L^\infty_x L_{mn}^2}\|\mathsf{D}^\alpha g\|_{L^2_xL_2^2}\|\langle z \rangle^{mn}\mathsf{D}^\alpha g\|_{L^2_xL^2} \lesssim |\!|\!|G|\!|\!|_N|\!|\!|g|\!|\!|_N^2.\]
		 If $n \ge N - 1$, then $|\alpha| \le 1$. Thus   if $N\ge4$
		\begin{equation*}
			\int|G|_{L_{mn}^2}|\mathsf{D}^\alpha g|_{L_2^2}|\langle z \rangle^{mn}\mathsf{D}^\alpha g|_{L^2}dx\le\|G\|_{L^2_xL_{mn}^2}\|\mathsf{D}^\alpha g\|_{L^\infty_x L_2^2}\|\langle z \rangle^{mn}\mathsf{D}^\alpha g\|_{L^2_xL^2} \lesssim |\!|\!|G|\!|\!|_N|\!|\!|g|\!|\!|_N^2.
		\end{equation*}
		This ends the proof.
	\end{proof}


	\subsection{Estimate of $J_2$ and $ J_3$.} We begin with a lemma on the commutator.
	\begin{lemma}\label{Comm2}
		For smooth functions $g= g(v), h = h(v)$ and $f = f(v)$ with $n \ge 4$, we have
		\begin{equation*}
		\begin{aligned}
			\langle \langle z \rangle^n\mathsf{Q}(g, h) - \mathsf{Q}(g, \langle z \rangle^nh), f \rangle_v &\lesssim (|g|_{L_4^2} + |g|_{H^2})|\langle z \rangle^nh|_{H_{\mathfrak{r}/2}^s}|f|_{L^2} + |g|_{L_n^2}|h|_{L_2^2}|f|_{L^2};\\
			\langle \langle z \rangle^n\mathsf{Q}(g, h) - \mathsf{Q}(g, \langle z \rangle^nh), f \rangle_v &\lesssim |g|_{L_4^2}(|\langle z \rangle^nh|_{H_{\mathfrak{r}/2}^s} + |\langle z \rangle^nh|_{H^1})|f|_{L^2} + |g|_{L_n^2}|h|_{L_2^2}|f|_{L^2}.\\
		\end{aligned}
		\end{equation*}
	\end{lemma}

	\begin{proof} Let $\mathsf{W}_n(v):= \langle z \rangle^n, \chi=\chi(v,v_*,\sigma):=1_{|v - v_*| \le 4|v_*|}1_{|v' - v| > 1}$. By the Taylor's expansion
		\begin{equation*}
			\langle z' \rangle^n - \langle z \rangle^n = (v' - v) \cdot \nabla \mathsf{W}_n(v') - \frac{1}{2}(v' - v)^\tau\int_0^1\nabla^2\mathsf{W}_n(\xi)d\kappa (v' - v),
		\end{equation*}
		where $\xi = v' + \kappa(v - v')$, we have the following decomposition:
		\begin{equation*}
			\mathfrak{J}:=\langle \langle z \rangle^n\mathsf{Q}(g, h) - \mathsf{Q}(g, \langle z \rangle^nh), f \rangle_v = \sum_{i=1}^4\mathfrak{J}_i,
		\end{equation*}
		 where
		\begin{equation*}
		\begin{aligned}
			\mathfrak{J}_1 &= \int g_*h'f'(v' - v) \cdot \nabla \mathsf{W}_n(v')B\chi d\sigma dvdv_*;
			\mathfrak{J}_2  = \int g_*(h - h')f'(v' - v) \cdot \nabla \mathsf{W}_n(v')B\chi d\sigma dvdv_*;\\
			\mathfrak{J}_3&= \int g_*hf'(v' - v)^\tau\nabla^2\mathsf{W}_n(\xi) (v' - v)B\chi d\kappa d\sigma dvdv_*; 
			\mathfrak{J}_4  = \int g_*hf'(\langle z' \rangle^n - \langle z \rangle^n)B(1-\chi)d\sigma dvdv_*.\\
		\end{aligned}
		\end{equation*}

		\noindent \noindent\underline{\it Estimate of $\mathfrak{J}_1, \mathfrak{J}_2$ and $\mathfrak{J}_4$.} We first remark that by Lemma \ref{lemsymetric}, $\mathfrak{J}_1=0$. The   estimates for $\mathfrak{J}_2$ and $\mathfrak{J}_4$ will follow by almost  the same way as those for $\mathfrak{I}_2$ and $\mathfrak{I}_4$ in Lemma \ref{Comm1}. Thus we conclude that
		\begin{equation*}
			|\mathfrak{J}_2|+|\mathfrak{J}_4| \lesssim |g|_{L_4^2}|\langle z \rangle^nh|_{L^2}|f|_{L^2} + |g|_{L_n^2}|h|_{L_2^2}|f|_{L^2}.
		\end{equation*}

	   \noindent \noindent\underline{\it Estimate of $\mathfrak{J}_2$.} Using the facts that $\langle z' \rangle^n(h - h') = \langle z \rangle^nh - \langle z' \rangle^nh' + (\langle z' \rangle^n - \langle z \rangle^n)h$, and $\langle v' \rangle \sim \langle v \rangle, \langle z' \rangle \sim \langle z \rangle$ on the support of $\chi$, we further decompose $\mathfrak{J}_2$ into two parts, i.e., $|\mathfrak{J}_2|\lesssim \mathfrak{J}_2^1+\mathfrak{J}_2^2$, where
		\beno
			\mathfrak{J}_2^1:= \int|g_*||\langle z \rangle^nh - \langle z' \rangle^nh'||\langle z' \rangle^{-1}f'||v' - v|B\chi d\sigma dvdv_*; \eeno 
			\beno \mathfrak{J}_2^2:= \int|g_*hf'||\langle z \rangle^n - \langle z' \rangle^n|\langle z' \rangle^{-1}|v' - v|B\chi d\sigma dvdv_*.\eeno

		By Cauchy-Schwartz inequality, we first get that
		\begin{equation*}
		\begin{aligned}
			\mathfrak{J}_2^1 &\lesssim \big(\int|g_*||\langle z \rangle^nh - \langle z' \rangle^nh'|^2Bd\sigma dvdv_*\big)^{1/2}   \big(\int|g_*||\langle z' \rangle^{-1}f'|^2|v' - v|^2B\chi d\sigma dvdv_*\big)^{1/2}:=(\mathfrak{J}_2^{1,1})^{\f12}(\mathfrak{J}_2^{1,2})^{\f12}.  
		\end{aligned}
		\end{equation*}
		Thanks to Corollary \ref{coerupper} and Lemma \ref{coerupper+}, we have
		\begin{equation*}
		\begin{aligned}
			\mathfrak{J}_2^{1,1}= D_{|g|}(\langle z \rangle^nh) \lesssim (|g|_{L_4^2} + |g|_{H^2})|\langle z \rangle^nh|_{H_{\mathfrak{r}/2}^s}^2; \quad
			\mathfrak{J}_2^{1,1}   \lesssim |g|_{L_4^2}(|\langle z \rangle^nh|_{H_{\mathfrak{r}/2}^s}^2 + |\langle z \rangle^nh|_{H^1}^2).\\
		\end{aligned}
		\end{equation*}
		Since $1-\chi \le 1_{|v - v_*| > 4|v_*|} + 1_{|v' - v| \le 1}$, following the notation  in  Lemma  \ref{singular},  we derive that 
		\begin{equation*}
		\begin{aligned}
			\mathfrak{J}_2^{1,2} &\lesssim \int|g_*||f'|^2\langle z \rangle^{-2}\langle z \rangle^{\gamma + 2}\theta^2b(\cos\theta)d\sigma dvdv_* +  \mathcal{J}_3(|g|, \langle z \rangle^{-1}f, \langle z \rangle^{-1}f)  \lesssim    |g|_{L_\mathfrak{r}^1}|f|_{L^2}^2.\\
		\end{aligned}
		\end{equation*}
		 We conclude that 
		\begin{equation*}
			\mathfrak{J}_2^1 \lesssim (|g|_{L_4^2} + |g|_{H^2})|\langle z \rangle^nh|_{H_{\mathfrak{r}/2}^s}|f|_{L^2},\quad \mathfrak{J}_2^1\lesssim |g|_{L_4^2}(|\langle z \rangle^nh|_{H_{\mathfrak{r}/2}^s} + |\langle z \rangle^nh|_{H^1})|f|_{L^2}.
		\end{equation*}
		
		To estimate $\mathfrak{J}_2^2$, we note that $|\langle z \rangle^n - \langle z' \rangle^n|\chi \lesssim \langle z \rangle^{n - 1}|v' - v|\chi$ since $\langle z' \rangle \chi \sim \langle z \rangle\chi$. Then 
		\begin{equation*}
			\mathfrak{J}_2^2 \lesssim \int|g_*\langle z \rangle^{n - 2}hf'|\theta^2|v - v_*|^{\gamma + 2}b(\cos\theta)\chi d\sigma dvdv_* \le \mathcal{I}_2(|g|, \langle z \rangle^n|h|, |f|) +\mathcal{J}_2(|g|, \langle z \rangle^{n - 2}|h|, |f|),
		\end{equation*} where we use the fact that
		 $1-\chi \le 1_{|v - v_*| > 4|v_*|} + 1_{|v' - v| \le 1}$ and follow the notations in Lemma \ref{kernel} and Lemma \ref{singular}. Therefore, we are led to that  $\mathfrak{J}_2^2\lesssim |g|_{L_4^2}|\langle z \rangle^nh|_{L^2}|f|_{L^2}.$
	 
		 Put together all the estimates, and we complete the proof.
	\end{proof}

	\begin{lemma}\label{ControlofJ2} For $N\ge5$, we have
		\begin{equation*}
			J_2 \lesssim |\!|\!|G|\!|\!|_N|\!|\!|g|\!|\!|_N^2.
		\end{equation*}
	\end{lemma}

	\begin{proof} Let $J_2=\sum\limits_{\substack{\alpha_1 + \alpha_2 = \alpha\\ |\alpha_1| > 0}} J_{2,\alpha_1,\alpha_2}$ with $J_{2,\alpha_1,\alpha_2}:=\langle \langle z \rangle^{mn}\mathsf{Q}(\mathsf{D}^{\alpha_1}G, \mathsf{D}^{\alpha_2}g) - \mathsf{Q}(\mathsf{D}^{\alpha_1}G, \langle z \rangle^{mn}\mathsf{D}^{\alpha_2}g), \langle z \rangle^{mn}\mathsf{D}^\alpha g \rangle.$
		Thanks to the second estimate in Lemma \ref{Comm2}, if $N\ge5$ and $\mathfrak{r}=\max\{\gamma+2s,0\}$, we claim that 
		\beno
			 |J_{2,\alpha_1,\alpha_2}|&\lesssim& \int|\mathsf{D}^{\alpha_1}G|_{L_4^2}|\langle z \rangle^{mn}\mathsf{D}^{\alpha_2}g|_{H_{\mathfrak{r}/2}^s \cap H^1}|\langle z \rangle^{mn}\mathsf{D}^\alpha g|_{L^2}dx + \int|\mathsf{D}^{\alpha_1}G|_{L_{mn}^2}|\mathsf{D}^{\alpha_2}g|_{L_2^2}|\langle z \rangle^{mn}\mathsf{D}^\alpha g|_{L^2}dx.\\  &\lesssim& |\!|\!|G|\!|\!|_N|\!|\!|g|\!|\!|_N^2.
		\eeno

		To prove it,  for the first term, if $|\alpha_1| \le N - 3$, by \eqref{TransDW1}, it is bounded by
			$$\|\mathsf{D}^{\alpha_1}G\|_{L^\infty_x L_4^2}\|\langle z \rangle^{mn}\mathsf{D}^{\alpha_2}g\|_{L^2_x(H_{\mathfrak{r}/2}^s \cap H^1)}\|\langle z \rangle^{mn}\mathsf{D}^\alpha g\|_{L^2_xL^2} \lesssim |\!|\!|G|\!|\!|_N|\!|\!|g|\!|\!|_N^2.$$
		And if $|\alpha_1| \ge N - 2$, by \eqref{TransDW1}, it is bounded by
		\begin{equation*}
			\|\mathsf{D}^{\alpha_1}G\|_{L^2_xL_4^2}\|\langle z \rangle^{mn}\mathsf{D}^{\alpha_2}g\|_{L^\infty_x(H_{\mathfrak{r}/2}^s \cap H^1)}\|\langle z \rangle^{mn}\mathsf{D}^\alpha g\|_{L^2_xL^2} \lesssim |\!|\!|G|\!|\!|_N|\!|\!|g|\!|\!|_N^2.
		\end{equation*}
		 The similar argument can be applied to the second term. Thus we end the proof.
	\end{proof}

	\begin{lemma}\label{ControlofJ3} 
		We have  $J_3 \lesssim |\!|\!|g|\!|\!|_N^2.$
	\end{lemma}

	\begin{proof} Let $J_3=\sum\limits_{\alpha_1 + \alpha_2 = \alpha} J_{3,\alpha_1,\alpha_2}$ with  $J_{3,\alpha_1,\alpha_2}:=\langle \langle z \rangle^{mn}\mathsf{Q}(\mathsf{D}^{\alpha_1}g, \mathsf{D}^{\alpha_2}\mathbf{M}) - \mathsf{Q}(\mathsf{D}^{\alpha_1}g, \langle z \rangle^{mn}\mathsf{D}^{\alpha_2}\mathbf{M}), \langle z \rangle^{mn}\mathsf{D}^\alpha g \rangle$.
		By Lemma \ref{Comm2}, we get that 
	\beno
			|J_{3,\alpha_1,\alpha_2}|  
			\lesssim \|\mathsf{D}^{\alpha_1}g\|_{L^2_xL_4^2}\|\langle z \rangle^{mn}\mathsf{D}^{\alpha_2}\mathbf{M}\|_{L^\infty_x H_1^1}\|\langle z \rangle^{mn}\mathsf{D}^\alpha g\|_{L^2_xL^2} \\
			+ \|\mathsf{D}^{\alpha_1}g\|_{L^2_xL_{mn}^2}\|\langle z \rangle^{mn}\mathsf{D}^{\alpha_2}\mathbf{M}\|_{L^\infty_x L_2^2}\|\langle z \rangle^{mn}\mathsf{D}^\alpha g\|_{L^2_xL^2} \lesssim |\!|\!|g|\!|\!|_N^2.
	\eeno 
		 This ends the proof.
	\end{proof}


	\subsection{Proof of Theorem \ref{GWPPRSSBE}: global well-posedness and scattering theory} Now we are in a position to give a detailed proof to result $(i)$ in Theorem \ref{GWPPRSSBE}.
	\smallskip

	\begin{proof}[Proof of Theorem \ref{GWPPRSSBE}: global well-posedness and scattering theory] We divide the proof into several steps.

	\underline{\it Step 1: A priori estimate.}   Apply  Lemma \ref{ControlofI1}, Lemma \ref{ControlofI2}, Lemma \ref{ControlofI34}, Lemma \ref{ControlofJ1}, Lemma \ref{ControlofJ2} and Lemma \ref{ControlofJ3} to   \eqref{EnergyG}, then we conclude that
	\beno 
\frac{1}{2}\frac{d}{dt}|\!|\!|g|\!|\!|_6^2 + \langle t \rangle^{-3 - \gamma}\sum_{n = 1}^6\sum_{|\alpha| \le 6 - n}\mathfrak{D}_G(\langle z \rangle^{mn}\mathsf{D}^\alpha g)\lesssim    \langle t \rangle^{-3 - \gamma}(|\!|\!|G|\!|\!|_6 + 1)|\!|\!|g|\!|\!|_6^2.\eeno
Using the fact that  $|\!|\!|G|\!|\!|_6 \le |\!|\!|g|\!|\!|_6 + |\!|\!|\mathbf{M}|\!|\!|_6 \lesssim |\!|\!|g|\!|\!|_6 + 1$, we further derive that 
	\beno \frac{1}{2}\frac{d}{dt}|\!|\!|g|\!|\!|_6^2 \lesssim \langle t \rangle^{-3 - \gamma}(|\!|\!|g|\!|\!|_6 + 1)|\!|\!|g|\!|\!|_6^2, \eeno
	which is enough to get that 
	\beno \sup_{t\in[0,\infty[} |\!|\!|g(t)|\!|\!|_6^2\lesssim |\!|\!|g_0|\!|\!|_6^2\lesssim \epsilon_0. \eeno  

	\underline{\it Step 2: Uniqueness.} Let  $G,H$ be two non-negative solutions to \eqref{DefiSSBES} with the same initial data.  Then $h:=H-G$ verifies that 
	\[\partial_th + \langle t \rangle^{-2}\mathsf{T}h = \langle t \rangle^{-3 - \gamma}[\mathsf{Q}(H, h) + \mathsf{Q}(h, G)].\]
	By Lemma \ref{ControlofI1}, Lemma \ref{ControlofJ1}, Lemma \ref{ControlofI34} and Lemma \ref{ControlofJ3}(with $n=N=1$), it is not difficult to get that 
	\beno \frac{1}{2}\frac{d}{dt}\|\langle z \rangle^m h\|_{L^2_xL^2}^2 + \langle t \rangle^{-3 - \gamma}\mathfrak{D}_{G + h}(\langle z \rangle^m h) \lesssim  \langle t \rangle^{-3 - \gamma}(|\!|\!|H|\!|\!|_6 + |\!|\!|G|\!|\!|_6)\|\langle z \rangle^mh\|_{L^2_xL^2}^2.\eeno  From this together with \eqref{EstGWP} for $G$ and $H$, we 
	 immediately deduce the uniqueness thanks to the Gronwall inequality.

	 \underline{\it Step 3: Non-negativity and local existence.} We first address that the non-cutoff Boltzmann collision operator $\mathsf{Q}$ can be defined via cutoff approximation(see \cite{HJZ}). As a matter of fact, if 
	 \beno  \mathsf{Q}_\eta(g,h):=\int B(g_*'h'-g_*h)\mathrm{1}_{\sin(\theta/2)\ge\eta}d\sigma dv_*,  \eeno
	 then for any $g\in L^1_{\gamma+2}, h\in L^2_{\gamma+2}$, the following holds in the sense of tempered distributions:
	 \ben\label{Qap}\mathsf{Q}(g,h)=\lim_{\eta\rightarrow0} \mathsf{Q}_\eta(g,h).\een
	 
	  Motivated by \eqref{Qap}, we introduce the following approximating equation to \eqref{DefiSSBES}, i.e.,
      \ben\label{APSSBE1}  \left \{
\begin{array}{lr}
\eta\in]0,1],\,G_\eta^0=G_\eta^0(t,x,v):=\mathbf{M},\, G_\eta^n=G_\eta^n(t,x,v);\\
\pa_t G^n_\eta+\lr{t}^{-2}\mathsf{T}G^n_\eta=\lr{t}^{-3-\gamma}\mathsf{Q}_\eta(G^{n-1}_\eta, G^n_\eta);\\
G^{n}_\eta|_{t=0}=G_0\ge0.
\end{array}	
\right.
 \een 
  Because of the cutoff on the deviation angle, the equation \eqref{APSSBE1} is a hyperbolic system and $\mathsf{Q}_\eta$ can be decomposed into the gain term and lost term. Thus, (i). by Duhamel's principle, the non-negativity of $G^n$ follows the non-negativity of $G^{n-1}$ and the initial data $G_0$; (2). by the estimates in {\it Step 1}, one may obtain that there exists a universal lifespan $T_*\le1$(independent of $n$ and $\eta$) such that for $n\in\N$,
	 \beno \sup_{t\in[0,T_*]}|\!|\!|G_\eta^{n}(t)|\!|\!|_6^2\le  |\!|\!|\mathbf{M}|\!|\!|_6^2+2. \eeno 
These are enough to get the   non-negative  and local solution   to the following Cauchy problem
     \ben\label{APSSBE2}  \left \{
\begin{array}{lr}
\pa_t G_\eta+\lr{t}^{-2}\mathsf{T}G_\eta=\lr{t}^{-3-\gamma}\mathsf{Q}_\eta(G_\eta, G_\eta);\\
G|_{t=0}=G_0\ge0,
\end{array}	
\right.
 \een 
with the uniform-in-$\eta$ estimate that  $\sup_{t\in[0,T_*]}|\!|\!|G_\eta(t)|\!|\!|_6^2\lesssim 1.$ Applying the compactness argument to $\{G_{1/n}\}_{n\in\N}$, it is not difficult to get the  the non-negativity and local existence of solution to \eqref{DefiSSBES}. From this together with the estimates in {\it Step 1} and  {\it Step 2}, we complete the proof to the global well-posedness of \eqref{DefiSSBES} near $\mathbf{M}$.

 \underline{\it Step 4: Scattering theory.} We split the proof into two parts. The first part is to show 
 \ben\label{EstTQL1} \sup_{t\in[0,\infty[}\|\mathsf{T}G(t)\|_{L^1_{x,v}} + \|\mathsf{Q}(G, G)(t)\|_{L^1_{x,v}} \lesssim 1.\een
By the definition of $\mathsf{T}$, it is easy to see that $\|\mathsf{T}G(t)\|_{L^1_{x,v}}\lesssim |\!|\!|G(t)|\!|\!|_6\lesssim1$. For the collision operator $\mathsf{Q}$, we have the following decomposition: $\mathsf{Q}(G,G)=\sum_{i=1}^4\mathsf{Q}_i(G,G)$, where  
\beno && \mathsf{Q}_1(G,G):=\int (G_*'G'-G_*G)B\mathrm{1}_{|v'-v|\ge1}d\sigma dv_*;\, \mathsf{Q}_2(G,G):=\int  B(G' - G)(G'_* - G_*)\mathrm{1}_{|v'-v|\le1}d\sigma dv_*;\\
&&\mathsf{Q}_3(G,G):=\int BG_*(G' - G) \mathrm{1}_{|v'-v|\le1}d\sigma dv_*;\quad \mathsf{Q}_4(G,G):=\int  BG(G'_* - G_*)\mathrm{1}_{|v'-v|\le1}d\sigma dv_*.
 \eeno 
 Since $\mathsf{Q}_1$ is a cutoff collision operator, we directly have $\|\mathsf{Q}_1(G,G)\|_{L^1_{x,v}}\lesssim |\!|\!|G(t)|\!|\!|^2_6\lesssim1$. For $\mathsf{Q}_2$, by Taylor expansion for $(G' - G)$ and $(G'_* - G_*)$, we deduce that
 \beno  \|\mathsf{Q}_2(G,G)\|_{L^1_{x,v}}\le \int |\na G|(\xi_1)|\na G|(\xi_2)|v-v_*|^{\gamma+2}\theta^2 d\sigma dv_*dvdxd\kappa d\kappa_1, \eeno 
where $\xi=\kappa v+(1-\kappa)v'$ and $\xi_1=\kappa_1 v_*+(1-\kappa_1)v'_*$. Thanks to \eqref{xivchange} and \eqref{xiv*change}, we derive that $\|\mathsf{Q}_2(G,G)\|_{L^1_{x,v}}\lesssim |\!|\!|G(t)|\!|\!|^2_6\lesssim1$. We claim that  $\|\mathsf{Q}_3(G,G)\|_{L^1_{x,v}}+\|\mathsf{Q}_3(G,G)\|_{L^1_{x,v}}\lesssim |\!|\!|G(t)|\!|\!|^2_6\lesssim1$. This follows the Taylor expansion up to the second order of $G'-G$ and $G'_*-G_*$. We skip the details here and complete the proof of the claim.

From the equation \eqref{DefiSSBES}, we have 
\beno G(t) = G(0) - \int_0^t\langle \mathsf{s} \rangle^{-2}(\mathsf{T}G)(\mathsf{s})d\mathsf{s} + \int_0^t\langle \mathsf{s} \rangle^{-3 - \gamma}\mathsf{Q}(G, G)(\mathsf{s})d\mathsf{s}. \eeno 
Thanks to \eqref{EstTQL1}, we deduce that 
\ben\label{DefiGinfty} G_\infty=G_\infty(x,v):= G(0) - \int_0^\infty \langle \mathsf{s} \rangle^{-2}(\mathsf{T}G)(\mathsf{s})d\mathsf{s} + \int_0^\infty\langle \mathsf{s} \rangle^{-3 - \gamma}\mathsf{Q}(G, G)(\mathsf{s})d\mathsf{s}\in L^1_{x,v}. \een
Moreover, $G\in C([0,\infty];L^1_{x,v})$ and it holds that 
\beno \|G(t)-G_\infty\|_{L^1_{x,v}}\le \int_t^\infty \langle \mathsf{s} \rangle^{-2}\|\mathsf{T}G)(\mathsf{s})\|_{L^1_{x,v}}d\mathsf{s} + \int_t^\infty\langle \mathsf{s} \rangle^{-3 - \gamma}\|\mathsf{Q}(G, G)(\mathsf{s})\|_{L^1_{x,v}}d\mathsf{s}\lesssim \max\{\lr{t}^{-1}, \lr{t}^{-\gamma-2}\}.\eeno 
	This ends the proof of the scattering theory. \end{proof} 
 

\section{Proof of Theorem \ref{GWPPRSSBE}: propagation of regularity}
	In this section, we will prove the propagation of analytic smoothness of the equation \eqref{DefiSSBES}. We address   that the following restrictions hold:
	\begin{equation*}
		s \in (0, \frac{1}{3}), \:\:\: \gamma \in ]-2, -2s]. \end{equation*}
To prove the desired result, we begin with several technical lemmas.

\begin{lemma}\label{lemTe1}
		There exists a constant $C = C(\delta)$ such that, for any $a_1, a_2, a, b \in \mathbb{N}$ satisfying that $a_1 + a_2 = a, 2a_1 \le a, b \le 4$, the following inequality holds:
		\begin{equation*}
			C_a^{a_1}\frac{1}{(a!)^{1 + \delta}} \le C(\delta)(a_1 + 1)^{-7/2}\frac{1}{((a_1 + b)!)^{1 + \delta}}\frac{1}{(a_2!)^{1 + \delta}}.
		\end{equation*}
	\end{lemma}

	\begin{proof}
		It is equivalent to prove that
		\ben\label{EstTec1}
			(a_1 + 1)^{7/2}\big[\frac{(a_1 + b)!}{a_1!}\big]^{1 + \delta} \lesssim (C_a^{a_1})^\delta.
		\een
		Obviously it holds for $a_1 = 0$. Let $a_1 \ge 1$. Since $b \le 4$, we have $(L.H.S.)$ of \eqref{EstTec1} is bounded by
$(1+a_1)^{8 + 4\delta}$.  Using Stirling's formula, one may derive that for $b>1$, one may derive that 
		\begin{equation*}
			\lim_{b \to \infty}\frac{(C_{2b}^b)^\delta}{(1+b)^{8 + \frac{4}{\delta}}} \gtrsim \lim_{b \to \infty}\big(\frac{\sqrt{b}(2b/e)^{2b}}{b(b/e)^{2b}}\big)^\delta \frac{1}{(1+b)^{8 + \frac{4}{\delta}}} \gtrsim \lim_{b \to \infty}\frac{2^{2b\delta}}{b^{8 + \delta + \frac{4}{\delta}}} = +\infty.
		\end{equation*}
		Then \eqref{EstTec1} follows since $C_a^{a_1} \ge C_{2a_1}^{a_1}$.  We end the proof.
	\end{proof}

    To prove the propagation of smoothness, we further introduce the function spaces as follows:
    \ben  &&\label{FSBN} \|f(t)\|_{BN}^2:=\sum_{\alpha\in\N^6}\|f(t)\|_{B_\alpha}^2: = \sum_{\alpha\in\N^6}(|\alpha| + 1)\frac{q(t)^{2(|\alpha| + 1)}}{(|\alpha|!)^{2 + 2\delta}}\|\langle z \rangle^4\mathsf{D}^\alpha f(t)\|_{L^2_xL^2}^2;\\
    &&\label{FSZ} \|f(t)\|_{Z}^2:=\sum_{\alpha\in\N^6}\|f(t)\|^2_{Z_\alpha} :=\sum_{\alpha\in\N^6} \frac{q(t)^{2(|\alpha| + 1)}}{(|\alpha|!)^{2 + 2\delta}}\|\langle z \rangle^4\mathsf{D}^\alpha f(t)\|_{L^2_xH^s}^2;\\
    &&\label{FSD}\|f(t)\|_{D}^2:=\sum_{\alpha\in\N^6}\|f(t)\|_{D_\alpha}:= \sum_{\alpha\in\N^6}(|\alpha| + 1)\frac{p(t)^{2(|\alpha| + 1)}}{(|\alpha|!)^{2 + 2\delta}}\|\langle z \rangle^4\mathsf{D}^\alpha f(t)\|_{L^2_xL^2}^2.\een
 To simply  \eqref{FSAN} and \eqref{FSMA}, we also define: 
 \ben\label{FSAMalpha}
    \|f(t)\|_{A_\alpha}^2:= \frac{q(t)^{2(|\alpha| + 1)}}{(|\alpha|!)^{2 + 2\delta}}\|\langle z \rangle^4\mathsf{D}^\alpha f(t)\|_{L^2_xL^2}^2,\,
   \|f(t)\|^2_{M_\alpha} = \frac{p(t)^{2(|\alpha| + 1)}}{(|\alpha|!)^{2 + 2\delta}}\|\langle z \rangle^4\mathsf{D}^\alpha f(t)\|_{L^2_xL^2}^2.
		\een
 
 Next we have 
	\begin{lemma}\label{lemTe2} For $\eta\ll1$, 
		it holds that $\|g(t)\|^2_Z \lesssim \eta\|g(t)\|_{BN}^2+ C_\eta\|g\|_{AN}^2$.
	\end{lemma}
	\begin{proof}
		Using   interpolation inequality, we get that
		\beno
			&&\|g\|_Z^2  = \sum_{\alpha\in\N^6}\frac{q(t)^{2(|\alpha| + 1)}}{(|\alpha|!)^{2 + 2\delta}}\|\langle z \rangle^4\mathsf{D}^\alpha g\|_{L^2_xH^s}^2 \lesssim \sum_{\alpha\in\N^6}\frac{q(t)^{2(|\alpha| + 1)}}{(|\alpha|!)^{2 + 2\delta}}\|\langle z \rangle^4\mathsf{D}^\alpha g\|_{L^2_xL^2}^{2 - 2s}\|\langle z \rangle^4\mathsf{D}^\alpha g\|_{L^2_xH^1}^{2s}\\ 
			&&\lesssim \big(\sum_{\alpha\in\N^6}(|\alpha| + 1)\frac{q(t)^{2(|\alpha| + 1)}}{(|\alpha|!)^{2 + 2\delta}}\|\langle z \rangle^4\mathsf{D}^\alpha g\|_{L^2_xL^2}^2\big)^{1 - s} \times \big(\sum_{\alpha\in\N^6}(|\alpha| + 1)^{-\f{1 - s}{s}}\frac{q(t)^{2(|\alpha| + 1)}}{(|\alpha|!)^{2 + 2\delta}}\|\langle z \rangle^4\mathsf{D}^\alpha g\|_{L^2_xH^1}^2\big)^s\\
			&&= \|g\|_{BN}^{2(1 - s)} \times \big(\sum_{\alpha\in\N^6}\sum_{|\beta| \le 1}(|\alpha| + 1)^{-\f{1 - s}{s}}\frac{q(t)^{2(|\alpha| + 1)}}{(|\alpha|!)^{2 + 2\delta}}\|\langle z \rangle^4\mathsf{D}^{\alpha + \beta}g\|_{L^2_xL^2}^2\big)^s.\\
		\eeno
        If $|\beta| = 1$, then from the facts that 
			$q(t)^{2|\beta|} \sim 1, \:\:\: \frac{(|\alpha + \beta|)!}{|\alpha|!} = |\alpha| + 1, \:\:\: 2 + 2\delta \le \frac{1 - s}{s}$,
		we conclude that 
		\begin{equation*}
		\begin{aligned}
			&\sum_{\alpha\in\N^6}\sum_{|\beta| = 1}(|\alpha| + 1)^{-\f{1 - s}{s}}\frac{q(t)^{2(|\alpha| + 1)}}{(|\alpha|!)^{2 + 2\delta}}\|\langle z \rangle^4\mathsf{D}^{\alpha + \beta}g\|_{L^2_xL^2}^2  \le \sum_{\alpha\in\N^6}\sum_{|\beta| = 1}\frac{q(t)^{2(|\alpha + \beta| + 1)}}{((|\alpha + \beta|)!)^{2 + 2\delta}}\|\langle z \rangle^4\mathsf{D}^{\alpha + \beta}g\|_{L^2_xL^2}^2 \le 6\|g\|_{AN}^2. 
		\end{aligned}
		\end{equation*}
      The desired result follows.
	\end{proof}

 Now we are in a position to prove $(ii)$ of Theorem \ref{GWPPRSSBE}.

 \begin{proof}[Proof of Theorem \ref{GWPPRSSBE}:Propagation of regularity.] We split the proof into two steps.
 \smallskip

\noindent\underline{\it Step 1: Propagation of the norm $\|\cdot\|_{AN}$.}	 Let $\alpha \in \mathbb{N}^6$. Recalling \eqref{Eq-g}, we first derive that
	\beno
		\partial_t[q(t)^{|\alpha| + 1}\langle z \rangle^4\mathsf{D}^\alpha g] + \langle t \rangle^{-2}\mathsf{T}[q(t)^{|\alpha| + 1}\langle z \rangle^4\mathsf{D}^\alpha g] - \frac{(|\alpha| + 1)q'(t)}{q(t)}q(t)^{|\alpha| + 1}
		\langle z \rangle^4\mathsf{D}^\alpha g \\= \langle t \rangle^{-3 - \gamma}q(t)^{|\alpha| + 1}\langle z \rangle^4\mathsf{D}^\alpha[\mathsf{Q}(G, g) + \mathsf{Q}(g, \mathbf{M})].
\eeno
	By basic energy method, one has
	\begin{multline*}
		\frac{1}{2}\frac{d}{dt}\frac{q(t)^{2(|\alpha| + 1)}}{(|\alpha|!)^{2 + 2\delta}}\|\langle z \rangle^4\mathsf{D}^\alpha g\|_{L_x^2L^2}^2
		- \frac{q'(t)}{q(t)}(|\alpha| + 1)\frac{q(t)^{2(|\alpha| + 1)}}{(|\alpha|!)^{2 + 2\delta}}\|\langle z \rangle^4\mathsf{D}^\alpha g\|_{L_x^2L^2}^2 \\
		= \langle t \rangle^{-3 - \gamma}\frac{q(t)^{2(|\alpha| + 1)}}{(|\alpha|!)^{2 + 2\delta}}\langle \mathsf{D}^\alpha[\mathsf{Q}(G, g) + \mathsf{Q}(g, \mathbf{M})], \langle z \rangle^8\mathsf{D}^\alpha g \rangle.
	\end{multline*} 
	Since 
 $\frac{3}{4} < q(t) \le 1, \:\:\: q'(t) = -(4T_\gamma)^{-1}\langle t \rangle^{-3 - \gamma}
 		$, by \eqref{FSBN} and \eqref{FSAMalpha}, it holds that
	\ben\label{EnerggAalpha}
		\frac{d}{dt}\|g\|^2_{A_\alpha} + C(\gamma)\langle t \rangle^{-3 - \gamma}\|g\|^2_{B_\alpha} \lesssim \langle t \rangle^{-3 - \gamma}\frac{q(t)^{2(|\alpha| + 1)}}{(|\alpha|!)^{2 + 2\delta}}\langle \mathsf{D}^\alpha[\mathsf{Q}(G, g) + \mathsf{Q}(g, \mathbf{M})], \langle z \rangle^8\mathsf{D}^\alpha g \rangle.
	\een

 \underline{Estimate of the term involving $\mathsf{Q}(G, g)$.} It can be reduced to consider following two terms:
	 \beno \mathfrak{R}_1:=\sum_{\alpha_1 + \alpha_2 = \alpha}C_\alpha^{\alpha_1}\langle \mathsf{Q}(\mathsf{D}^{\alpha_1}G, \langle z \rangle^4\mathsf{D}^{\alpha_2}g), \langle z \rangle^4\mathsf{D}^\alpha g \rangle;
	 \\\mathfrak{R}_2:=
		\sum_{\alpha_1 + \alpha_2 = \alpha}C_\alpha^{\alpha_1}\langle \langle z \rangle^4\mathsf{Q}(\mathsf{D}^{\alpha_1}G, \mathsf{D}^{\alpha_2}g) - \mathsf{Q}(\mathsf{D}^{\alpha_1}G, \langle z \rangle^4\mathsf{D}^{\alpha_2}g), \langle z \rangle^4\mathsf{D}^\alpha g \rangle.
	\eeno
	 
	Since  $\gamma + 2s \le 0$, by  Corollary \ref{HLB1}, we have 
	 
	\begin{equation*}
	|\mathfrak{R}_1|	\lesssim \sum_{\alpha_1 + \alpha_2 = \alpha}C_\alpha^{\alpha_1}\int|\langle z \rangle^4\mathsf{D}^{\alpha_1}G|_{L^2}|\langle z \rangle^4\mathsf{D}^{\alpha_2}g|_{H^s}|\langle z \rangle^4\mathsf{D}^\alpha g|_{H^s}dx.
	\end{equation*}
	Thanks to Lemma \ref{Comm2}, we deduce that 
	\begin{equation*}
	\begin{aligned}
		|\mathfrak{R}_2|\lesssim \sum_{\alpha_1 + \alpha_2 = \alpha}C_\alpha^{\alpha_1}\int|\langle z \rangle^4\mathsf{D}^{\alpha_1}G|_{H^2}|\langle z \rangle^4\mathsf{D}^{\alpha_2}g|_{H^s}|\langle z \rangle^4\mathsf{D}^\alpha g|_{L^2}dx; \\
		 |\mathfrak{R}_2|\lesssim \sum_{\alpha_1 + \alpha_2 = \alpha}C_\alpha^{\alpha_1}\int|\langle z \rangle^4\mathsf{D}^{\alpha_1}G|_{L^2}|\langle z \rangle^4\mathsf{D}^{\alpha_2}g|_{H^1}|\langle z \rangle^4\mathsf{D}^\alpha g|_{L^2}dx.\\
	\end{aligned}
	\end{equation*}

	These imply that 
	\begin{equation}\label{I01}
	\begin{aligned}
		|(\mathfrak{R}_1+\mathfrak{R}_2)| &\lesssim \sum_{\alpha_1 + \alpha_2 = \alpha}C_\alpha^{\alpha_1}\|\langle z \rangle^4\mathsf{D}^{\alpha_1}G\|_{L^\infty_x H^2}\|\langle z \rangle^4\mathsf{D}^{\alpha_2}g\|_{L^2_xH^s}\|\langle z \rangle^4\mathsf{D}^\alpha g\|_{L^2_xH^s};\\
		|(\mathfrak{R}_1+\mathfrak{R}_2)|  &\lesssim \sum_{\alpha_1 + \alpha_2 = \alpha}C_\alpha^{\alpha_1}\|\langle z \rangle^4\mathsf{D}^{\alpha_1}G\|_{L^2_xL^2}\|\langle z \rangle^4\mathsf{D}^{\alpha_2}g\|_{L^\infty_x H^{1}}\|\langle z \rangle^4\mathsf{D}^\alpha g\|_{L^2_xH^s}.\\
	\end{aligned}
	\end{equation}
	We separate  $(\mathfrak{R}_1+\mathfrak{R}_2)$ into two parts: $(\mathfrak{R}_1+\mathfrak{R}_2)_1$ and $(\mathfrak{R}_1+\mathfrak{R}_2)_2$, which correspond to cases  $2|\alpha_1| \le |\alpha|$  and $2|\alpha_1| > |\alpha|$ respectively. 
    Using the first inequality of \eqref{I01}, we get that  
	\begin{equation*}
	 |(\mathfrak{R}_1+\mathfrak{R}_2)_1| \lesssim \sum_{\substack{\alpha_1 + \alpha_2 = \alpha\\ 2|\alpha_1| \le |\alpha|}}\sum_{|\beta| \le 4}C_\alpha^{\alpha_1}\|\langle z \rangle^4\mathsf{D}^{\alpha_1 + \beta}G\|_{L^2_xL^2}\|\langle z \rangle^4\mathsf{D}^{\alpha_2}g\|_{L^2_xH^s}\|\langle z \rangle^4\mathsf{D}^\alpha g\|_{L^2_xH^s}.
	\end{equation*}
	Noting that $C_\alpha^{\alpha_1} \le C_{|\alpha|}^{|\alpha_1|},$ by Lemma \ref{lemTe1}, we deduce that
	\begin{equation}\label{I101}
		(|\alpha|!)^{-2 - 2\delta}|(\mathfrak{R}_1+\mathfrak{R}_2)_1|  \lesssim \sum_{\substack{\alpha_1 + \alpha_2 = \alpha\\ 2|\alpha_1| \le |\alpha|}}\sum_{|\beta| \le 4}\frac{\|\langle z \rangle^4\mathsf{D}^{\alpha_1 + \beta}G\|_{L^2_xL^2}}{(|\alpha_1 + \beta|!)^{1 + \delta}}
		\frac{\|\langle z \rangle^4\mathsf{D}^{\alpha_2}g\|_{L^2_xH^s}}{(|\alpha_2|!)^{1 + \delta}}\frac{\|\langle z \rangle^4\mathsf{D}^\alpha g\|_{L^2_xH^s}}{(|\alpha_1| + 1)^{7/2}(|\alpha|!)^{1 + \delta}}.
	\end{equation} 
	Since $q(t) \in (3/4, 1], |\beta| \le 4$, then $q(t)^{-|\beta| - 1} \lesssim 1$. Substituting it into \eqref{I101}, by \eqref{FSZ}, we finally get that
	\begin{equation*}
		\frac{q(t)^{2(|\alpha| + 1)}}{(|\alpha|!)^{2 + 2\delta}}|(\mathfrak{R}_1+\mathfrak{R}_2)_1|  \lesssim\sum_{\substack{\alpha_1 + \alpha_2 = \alpha\\ 2|\alpha_1| \le |\alpha|}}\sum_{|\beta| \le 4} \|G\|_{A_{\alpha_1 + \beta}}\|g\|_{Z_{\alpha_2}}\frac{\|g\|_{Z_\alpha}}{(|\alpha_1| + 1)^7},
	\end{equation*}
	which yields that 
	\begin{equation*}
	\begin{aligned}
		\sum_{\alpha\in\N^6}\frac{q(t)^{2(|\alpha| + 1)}}{(|\alpha|!)^{2 + 2\delta}}|(\mathfrak{R}_1+\mathfrak{R}_2)_1| &\lesssim   \|G\|_{AN}^2\|g\|^2_Z + \|g\|_Z^2.\\
	\end{aligned}
	\end{equation*}

	For  $(\mathfrak{R}_1+\mathfrak{R}_2)_2,$ we use the second inequality in \eqref{I01} to have
	\begin{equation*}
		|(\mathfrak{R}_1+\mathfrak{R}_2)_2| \lesssim \sum_{\substack{\alpha_1 + \alpha_2 = \alpha\\ 2|\alpha_1| > |\alpha|}}\sum_{|\beta| \le 3}C_\alpha^{\alpha_1}\|\langle z \rangle^4\mathsf{D}^{\alpha_1}G\|_{L^2_xL^2}\|\langle z \rangle^4\mathsf{D}^{\alpha_2 + \beta}g\|_{L^2_xH^s}\|\langle z \rangle^4\mathsf{D}^\alpha g\|_{L^2_xH^s}.
	\end{equation*}
	The similar argument can be applied to get the same estimate as $(\mathfrak{R}_1+\mathfrak{R}_2)_1$. Thus, we have   
	 \begin{equation*}
	\begin{aligned}
		\sum_{\alpha\in\N^6}\frac{q(t)^{2(|\alpha| + 1)}}{(|\alpha|!)^{2 + 2\delta}}|(\mathfrak{R}_1+\mathfrak{R}_2)| &\lesssim   \|G\|_{AN}^2\|g\|^2_Z + \|g\|_Z^2.\\
	\end{aligned}
	\end{equation*}

	 \underline{Estimate of the term involving $\mathsf{Q}(g, \mathbf{M})$.}   Let $\mathfrak{L}:=\langle \mathsf{D}^\alpha\mathsf{Q}(g, \mathbf{M}), \langle z \rangle^8\mathsf{D}^\alpha g \rangle$. Following the argument in the before, we deduce that 
	\begin{equation*}\sum_{\alpha\in\N^6} \frac{q(t)^{2(|\alpha| + 1)}}{(|\alpha|!)^{2 + 2\delta}}\mathfrak{L} \lesssim \|g\|_{AN}^2\|\mathbf{M}\|_Z^2 + \|g\|_Z^2.
	\end{equation*}
	\smallskip

	Putting these estimates together, we are led to that
	\begin{equation}\label{001}
		\frac{d}{dt}\|g\|_{AN}^2 + C(\gamma)\langle t \rangle^{-3 - \gamma}\|g\|_{BN}^2 \lesssim \langle t \rangle^{-3 - \gamma}(\|G\|_{AN}^2\|g\|_Z^2 + \|g\|_{AN}^2\|\mathbf{M}\|_Z^2 + \|g\|_Z^2).
	\end{equation}
	 Applying Lemma \ref{lemTe2}, we have
	\begin{equation*}
		\frac{d}{dt}\|g\|_{AN}^2 + C(\gamma)\langle t \rangle^{-3 - \gamma}\|g\|_{BN}^2 \lesssim C_\eta \langle t \rangle^{-3 - \gamma}(\|g\|_{AN}^2 + 1)\|g\|_{AN}^2 + \eta\langle t \rangle^{-3 - \gamma}(\|g\|_{AN}^2+1)\|g\|_{BN}^2.
	\end{equation*}
     From this together with the initial condition that $\|g(0)\|_{AN}^2=\epsilon_1\ll1$, we obtain the upper bound in \eqref{EstPRLB}. 
\smallskip

\noindent\underline{\it Step 2: Propagation of the norm $\|\cdot\|_{MA}$.}	 Let $\alpha \in \mathbb{N}^6$. We first derive that
\beno
			\partial_t[p(t)^{|\alpha| + 1}\langle z \rangle^4\mathsf{D}^\alpha g] + \langle t \rangle^{-2}\mathsf{T}[p(t)^{|\alpha| + 1}\langle z \rangle^4\mathsf{D}^\alpha g] - \frac{(|\alpha| + 1)p'(t)}{p(t)}p(t)^{(|\alpha| + 1)^{1 + \delta}}
			\langle z \rangle^4\mathsf{D}^\alpha g \\= \langle t \rangle^{-3 - \gamma}p(t)^{|\alpha| + 1}\langle z \rangle^4\mathsf{D}^\alpha[\mathsf{Q}(G, g) + \mathsf{Q}(g, \mathbf{M})].
		\eeno
		The basic energy implies that 
		\beno
			\frac{1}{2}\frac{d}{dt}\frac{p(t)^{2(|\alpha| + 1)}}{(|\alpha|!)^{2 + 2\delta}}\|\langle z \rangle^4\mathsf{D}^\alpha g\|_{L^2_xL^2}^2
			- \frac{p'(t)}{p(t)}(|\alpha| + 1)\frac{p(t)^{2(|\alpha| + 1)^{1 + \delta}}}{(|\alpha|!)^{2 + 2\delta}}\|\langle z \rangle^4\mathsf{D}^\alpha g\|_{L^2_xL^2}^2 \\
			= \langle t \rangle^{-3 - \gamma}\frac{p(t)^{2(|\alpha| + 1)}}{(|\alpha|!)^{2 + 2\delta}}\langle \mathsf{D}^\alpha[\mathsf{Q}(G, g) + \mathsf{Q}(g, \mathbf{M})], \langle z \rangle^8\mathsf{D}^\alpha g \rangle.
		\eeno

	Noting that $\frac{1}{4} \le p(t) < \frac{1}{2}, \:\:\: p'(t) = (4T_\gamma)^{-1}\langle t \rangle^{-3 - \gamma}$, by \eqref{FSD} and \eqref{FSAMalpha},  we get that	
\ben\label{EnergygMalpha}
			\frac{d}{dt}\|g\|^2_{M_\alpha} - C(\gamma)\langle t \rangle^{-3 - \gamma}\|g\|_{D_\alpha}^2 \gtrsim
			\langle t \rangle^{-3 - \gamma}\frac{p(t)^{2(|\alpha| + 1)}}{(|\alpha|!)^{2 + 2\delta}}\langle \mathsf{D}^\alpha[\mathsf{Q}(G, g) + \mathsf{Q}(g, \mathbf{M})], \langle z \rangle^8\mathsf{D}^\alpha g \rangle.
		\een
We observe that the (R.H.S) of  \eqref{EnergygMalpha} enjoys almost the same structure as (R.H.S) of  \eqref{EnerggAalpha}. Thus by repeating the argument used in {\it Step 1}, we may derive that 
\beno \frac{d}{dt}\|g\|^2_{MA} - C(\gamma)\langle t \rangle^{-3 - \gamma}\|g\|_{D}^2 \gtrsim -C_\eta\langle t \rangle^{-3 - \gamma}(\|g\|^2_{MA} + 1)\|g\|_{MA}^2 - \eta\langle t \rangle^{-3 - \gamma}(\|g\|_{MA}^2+1)\|g\|_{D}^2. \eeno 
Since $\|g\|^2_{MA}\le \|g\|^2_{AN}\ll1$, we finally conclude that
\begin{equation*}
			\frac{d}{dt}\|g\|_{MA}^2 \gtrsim -\langle t \rangle^{-3 - \gamma}\|g\|_{MA}^2,
		\end{equation*}
		which implies the lower bound in \eqref{EstPRLB}. 
		\smallskip

\underline{\it Step 3: $G_\infty\neq \mathbf{M}$.} 
	 Suppose $G_\infty = \mathbf{M}$. Since now $\gamma+2s\le0$ and $g=G-\mathbf{M}=G-G_\infty$, by repeating the argument in proof of Theorem \ref{GWPPRSSBE}(in particular, in ${\it Step 4}$), one may further have    
	 \[\lim_{t\rightarrow\infty}\|\langle z \rangle^4 g(t)\|_{L^1_{x,v}}=0.\]  Next we claim that for any fixed $N \in \mathbb{N}$, it holds that
	  \begin{equation}\label{tinftyNMA}
			\lim_{t \to \infty}\sum_{|\alpha| \le N}\frac{p(t)^{2(|\alpha| + 1)}}{(|\alpha|!)^{2 + 2\delta}}\|\langle z \rangle^4\mathsf{D}^\alpha g\|_{L_x^2L^2}^2 = 0.
		\end{equation}
		This  follows \eqref{Upperxv} and the computation that
		\beno
			\sum_{|\alpha| \le N}\frac{p(t)^{2(|\alpha| + 1)}}{(|\alpha|!)^{2 + 2\delta}}\|\langle z \rangle^4\mathsf{D}^\alpha g(t)\|_{L_x^2L^2}^2 \lesssim 
			C(N)\|\langle z \rangle^4g(t)\|_{H_{x, v}^{2N + 4}}\|\langle z \rangle^4g(t)\|_{H_{x, v}^{-4}} \le C(N)\|g(t)\|_{AN}\|\langle z \rangle^4g(t)\|_{L_{x, v}^1}.
		\eeno

		Thanks to Theorem \ref{GWPPRSSBE}, and the facts that $p(t) \in [1/4, 1/2), q(t) \in (3/4, 1]$, we have
		\begin{equation*}
			\sum_{|\alpha| > N}\frac{p(t)^{2(|\alpha| + 1)}}{(|\alpha|!)^{2 + 2\delta}}\|\langle z \rangle^4\mathsf{D}^\alpha g(t)\|_{L_x^2L^2}^2 = \sum_{|\alpha| > N}\big(\frac{p(t)}{q(t)}\big)^{2(|\alpha| + 1)}\frac{q(t)^{2(|\alpha| + 1)}}{(|\alpha|!)^{2 + 2\delta}}\|\langle z \rangle^4\mathsf{D}^\alpha g(t)\|_{L_x^2L^2}^2 \lesssim (\frac{2}{3})^{2N}.
		\end{equation*}
		From this together with \eqref{tinftyNMA}, we derive that there exist constants $N_1,t_*\gg1$ such that for any $t\ge t_*$,
        \beno \|g(t)\|_{MA}^2=\sum_{|\alpha| > N_1}\frac{p(t)^{2(|\alpha| + 1)}}{(|\alpha|!)^{2 + 2\delta}}\|\langle z \rangle^4\mathsf{D}^\alpha g\|_{L_x^2L^2}^2+ \sum_{|\alpha| \le N_1}\frac{p(t)^{2(|\alpha| + 1)}}{(|\alpha|!)^{2 + 2\delta}}\|\langle z \rangle^4\mathsf{D}^\alpha g\|_{L_x^2L^2}^2\lesssim \epsilon_2^5,\eeno 
        which contradicts with \eqref{EstPRLB}. We conclude the desired result and end the proof.
	\end{proof}


	\section{Proof of Theorem \ref{StabilityBE}}

	This section is devoted to the proof of Theorem \ref{StabilityBE}.  Thanks to \eqref{GtoF2S} and \eqref{FandG}, result $(i)$ in Theorem \ref{GWPPRSSBE} will immediately imply result $(i)$ in Theorem \ref{StabilityBE}. Thus, we only need to give a detailed proof to result $(ii)$ in Theorem \ref{StabilityBE}.

   \begin{proof}[Proof of Theorem \ref{StabilityBE}: result $(ii)$.] We prove it by contradiction argument. Suppose \eqref{FMdistance} does not hold. Since for any $t\ge0$, $\|(F - \mathcal{M})(t)\|_{L^1_{x,v}} = \|(G - \mathbf{M})(t)\|_{L^1_{x,v}}$, it implies that there exists a sequence $\{t_n\}_{n\in\N}$ such that $ \|(G - \mathbf{M})(t_n)\|_{L^1_{x,v}}\le 1/n$. We claim that this implies that $G_\infty=\mathbf{M}$ which contradicts with the result $G_\infty\neq\mathbf{M}$. Thus we get \eqref{FMdistance}. Now it suffices to prove the claim. If $\sup_n \{t_n\}<\infty$, then by the fact that $G\in C([0,\infty];L^1_{x,v})$, we deduce that there exists a finite time $T_*$ such that $G(T_*)=\mathbf{M}$. By $H$-theorem stated in Theorem \ref{ConsHthSSBE}, we have for any $t\ge T_*$,   $\mathcal{H}[G|\mathbf{M}](t)=0$, which in turn implies that $G_\infty=\mathbf{M}$. If $\sup_n \{t_n\}=\infty$, then by scattering theory \eqref{scatteringL1}, we also derive that  $G_\infty=\mathbf{M}$. This ends the proof.
\end{proof}

	\section{Appendix}
  The Appendix consists of two parts. The first part provides a detailed proof for Proposition \ref{CharacterTM}. The second part is devoted to the list of the knowledge on the basic properties on the non-cutoff equation.    

\subsection{Proof of Proposition \ref{CharacterTM}} We first remark that Proposition \ref{CharacterTM} is given in the preprint paper \cite{LM}. To make our paper self-contained, we provide a detailed proof  but the key idea are based on \cite{LM}. 
\smallskip

To derive Proposition \ref{CharacterTM}, we only need to show

\begin{lemma}\label{LemforM2}  Let  $a, b, c, d \in \mathbb{R}, p \in \mathbb{R}^6$ and $A$ is a $3 \times 3$ skew-symmetric matrix. Suppose that $\mathcal{M}_2(t, x, v) = \exp\{-\frac{1}{2}[a|v|^2 + 2bv \cdot (x - tv) + c|x - tv|^2 + 2v^\tau Ax + p \cdot (v, x - tv)] + d\}$ satisfies $\Phi_1(\mathcal{M}_2)=1$ and $\Phi_2(\mathcal{M}_2 )=\Phi_3(\mathcal{M}_2)=0$. Then it hold that
\begin{itemize}
\item  $p=0,a, c, ac - b^2 > 0$, $P$ defined in \eqref{DefzP} is positive definite and $e^d = \frac{\sqrt{\det P}}{(2\pi)^3}$;
\item  If  $Q:= (ac - b^2)I + A^2$, then  $Q$ is also positive definite and $\det P = \det Q$;
\item \begin{equation*}
		\mathcal{M}_2(t, x, v) = \frac{\sqrt{\det Q}}{(2\pi)^3}\exp\bigg\{-\frac{1}{2}\big[a|v|^2 + 2bv \cdot (x - tv) + c|x - tv|^2 + 2v^\tau Ax\big]\bigg\}.
	\end{equation*}
\end{itemize} 
	\end{lemma}

 \begin{proof}    If we set \begin{equation}\label{DefzP}
		z:= \begin{pmatrix} v \\ x \end{pmatrix} \in \mathbb{R}^6, \:\:\:\:\:\: P:=
		\begin{pmatrix} aI & bI + A \\ bI - A & cI \end{pmatrix}
	\end{equation}
	then it holds that
 $a|v|^2 + 2bv \cdot x + c|x|^2 + 2v^\tau Ax =  z^\tau Pz.$
	Since $\Phi_1(\mathcal{M}_2)=1$ and $\Phi_2(\mathcal{M}_2 )=\Phi_3(\mathcal{M}_2)=0$., we get that
	\begin{equation}\label{ProptypP}
		\int_{\R^6} e^{-\frac{1}{2}(z^\tau Pz + p \cdot z) + d}dz = 1, \:\:\:\:\:\: \int_{\R^6} ze^{-\frac{1}{2}(z^\tau Pz + p \cdot z) + d}dz = 0.
	\end{equation}

 First we prove $p=0$.  To see it, we introduce   integral region $I = \{z \in \R^6 : p \cdot z \ge 0\}$ and $J = \{z \in \R^6 : p \cdot z \le 0\}$. It is easy to check that $J = -I$ and moreover by \eqref{ProptypP},
\beno
		0 &=& \int_Ip \cdot ze^{-\frac{1}{2}(z^\tau Pz + p \cdot z) + d}dz + \int_Jp \cdot ze^{-\frac{1}{2}(z^\tau Pz + p \cdot z) + d}dz \\
		&=& \int_I p \cdot ze^{-\frac{1}{2}(z^\tau Pz + d)} (e^{-\frac{1}{2}p \cdot z}-e^{ \frac{1}{2}p \cdot z}) dz \le 0, \eeno
	which implies that $p \cdot z = 0$,  for a.e. $z \in \R^6$. Thus we have $p = 0$.

	By definition \eqref{DefzP}, $P$ is a symmetric matrix. Then there exists a orthogonal matrix $C$ and the diagonal matrix $D := \mathrm{diag}\{\lambda_1, \cdots, \lambda_6\}$ such that $P = C^\tau DC$. Use change of variable in \eqref{ProptypP}, then we  get that
	\begin{equation*}
		1 = \int_{\R^6} e^{-\frac{1}{2}z^\tau C^\tau DCz + d}dz = \int_{\R^6} e^{-\frac{1}{2}z^\tau Dz + d}dz=e^d\int_{\R^6} \prod_{i = 1}^6e^{-\frac{1}{2}\lambda_iz_i^2}dz.
	\end{equation*}
	This forces to have $\lambda_i>0$ for $1\le i\le 6$ which implies that  $P$ is positive definite. Besides, by the fact that $\int_{\R^3}e^{-\frac{1}{2}\lambda_iz_i^2}dz_i=\lambda_i^{-\f12}\sqrt{2\pi}$,  we have $e^d = \frac{\sqrt{\det P}}{(2\pi)^3}$.   Recall $z^\tau Pz = a|v|^2 + 2bv \cdot x + c|x|^2 + 2v^\tau Ax > 0$  and $P$ is positive definite. If we take $v = kx$ with $|x| = 1$ and use the fact that  $A$ is skew-symmetric,  then $v^\tau Ax = kx^\tau Ax = 0$  and $ak^2 + 2bk + c > 0$ for all $k \in \R$. Thus we have  $a, c, ac - b^2 > 0$.

	By definition of $P$, it is not difficult to check that $\det P = \det[(ac-b^2)I + A^2] = \det Q$. To show that $Q$ is also positive definite,  if $v = -(bx + Ax)/a$, then we have
	\begin{equation*}
	\begin{aligned}
		z^\tau Pz &= \frac{|bx + Ax|^2}{a} - \frac{2(bx + Ax) \cdot bx}{a} + c|x|^2 - \frac{2(bx + Ax)^\tau Ax}{a} \\
		&= \frac{1}{a}(b^2|x|^2 + 2bx \cdot Ax + |Ax|^2 - 2b^2|x|^2 - 2bx \cdot Ax + ac|x|^2 - 2|Ax|^2) \\
		&= \frac{1}{a}[(ac - b^2)|x|^2 + x^\tau A^2x] = \frac{x^\tau Qx}{a},
	\end{aligned}
	\end{equation*}
which implies our desired result.
\end{proof}

\subsection{Basic properties of non-cutoff collision operator} We list some basic tools for the non-cutoff equation. We begin with a lemma on the change of variables.

\begin{lemma}\label{chv} (\cite{ADVW}) We have  \\
(1) (Regular change of variables)
\[\int F(v', |v - v_*|, \theta)d\sigma dv = \int\frac{1}{\cos^3(\theta/2)}F(v, \frac{|v - v_*|}{\cos(\theta/2)}, \theta)d\sigma dv.
\]
(2) (Singular change of variables)
\[\int F(v', |v - v_*|, \theta)d\sigma dv_* = \int\frac{1}{\sin^3(\theta/2)}F(v_*, \frac{|v - v_*|}{\sin(\theta/2)}, \theta)d\sigma dv_*.
\]
\end{lemma}

\begin{lemma}\label{Cancellation} (Cancellation Lemma, \cite{ADVW} Lemma 1) For any smooth function $f$, it holds that
$\int_{\R^3 \times \mathbb{S}^{2}} B(v-v_*, \sigma) (f'-f) dv d\sigma= (f*S )(v_*),$
where
\[
S(z) = |\mathbb{S}^1| \int_{0}^{\frac \pi 2} \sin \theta \left[ \frac 1 {\cos^3(\theta/2)} B\left(\frac {|z|}  {\cos (\theta/2)} , \cos \theta \right) - B(|z|, cos \theta) \right].
\]
\end{lemma}

\begin{lemma}[\cite{ADVW}]\label{lemsymetric} It hold that
\beno 
			 \int(v' - v)F(v', |v - v_*|, \theta)d\sigma dv = 0; \quad
  \int(v' - v)f(\theta)d\sigma = -(v - v_*)\int\sin^2\frac{\theta}{2}f(\theta)d\sigma.
		\eeno
\end{lemma}
Next we focus on some estimates on the collision operator.
	\begin{lemma}\label{HLB}
		[\cite{H}]. For $a, b \in [0, 2s], w_1, w_2 \in \mathbb{R}$ with $a + b = 2s, w_1 + w_2 = \gamma + 2s$. Then for smooth functions $g = g(v), h = h(v)$ and $f = f(v)$, we have \\
		(1)if $ \gamma + 2s > 0$
		\begin{equation*}
			|\langle \mathsf{Q}(g, h), f \rangle_v| \lesssim (|g|_{L_{\gamma + 2s + w}^1} + |g|_{L^2})|h|_{H_{w_1}^a}|f|_{H_{w_2}^b};
		\end{equation*}
		where $w = (-w_1)^+ + (-w_2)^+$ \\
		(2) if $\gamma + 2s = 0$
		\begin{equation*}
			|\langle \mathsf{Q}(g, h), f \rangle_v| \lesssim_\delta (|g|_{L_{w_3}^1} + |g|_{L^2})|h|_{H_{w_1}^a}|f|_{H_{w_2}^b};
		\end{equation*}
		where $\delta > 0, w_3 = \max\{\delta, w\}$ \\
		(3) if $-1 < \gamma + 2s < 0$
		\begin{equation*}
			|\langle \mathsf{Q}(g, h), f \rangle_v| \lesssim (|g|_{L_{w_4}^1} + |g|_{L_{|\gamma + 2s|}^2})|h|_{H_{w_1}^a}|f|_{H_{w_2}^b},
		\end{equation*}
		where $w_4 = \max\{|\gamma + 2s|, \gamma + 2s + w\}$ \\
	\end{lemma}

	\begin{corollary}\label{HLB1}
		Let $a, b, w_1, w_2 $ satisfy the conditions  in lemma \ref{HLB} with $w_1w_2 \ge 0$. There holds
		\begin{equation*}
			|\langle \mathsf{Q}(g, h), f \rangle_v| \lesssim |g|_{L_4^2}|h|_{H_{w_1}^a}|f|_{H_{w_2}^b}.
		\end{equation*}
	\end{corollary}

	\begin{lemma}\label{coerupper+}
		It holds that
		\begin{equation*}
			\int g_*f^2|v - v_*|^\gamma dvdv_* \lesssim |g|_{L_{|\gamma|}^1}|f|_{H_{\gamma/2}^{|\gamma|/2}}^2,
		\end{equation*}
		from this together with \eqref{coerandQ} and Corollary \ref{HLB1}, we get that $D_g(f) \lesssim |g|_{L_4^2}(|f|_{H_{r/2}^s}^2 + |f|_{H^1}^2).$
	\end{lemma}

	\begin{lemma}\label{v'bound}
	 The following inequality  holds: $|v - v_*| \lesssim \langle v_* \rangle\langle v' \rangle$.
	\end{lemma}

	\begin{proof}
		If $|v| < 3|v_*|$, then $|v - v_*| \le 4|v_*| \le 4\langle v_* \rangle\langle v' \rangle$. If $|v| \ge 3|v_*|$, then $|v - v_*| \le 4|v|/3$. We get that $|v'| \gtrsim |v|$ since $v' = \frac{v + v_*}{2} + \frac{|v - v_*|}{2}\sigma$ and
		\begin{equation*}
		\begin{aligned}
			|v'|^2 &= \frac{|v|^2 + |v_*|^2}{2} + \frac{(v + v_*) \cdot \sigma}{2}|v - v_*| = \frac{|v|^2 + |v_*|^2}{2} + \frac{\cos\theta}{2}|v - v_*|^2 + v_* \cdot \sigma|v - v_*| \\
			&\ge \frac{|v|^2}{2}- |v_*||v - v_*| \ge \frac{1}{18}|v|^2 + \frac{4}{9}|v|^2- \frac{4}{3}|v_*||v| \ge \frac{1}{18}|v|^2.
		\end{aligned}
		\end{equation*}
		We end the proof.
	\end{proof}

	\begin{lemma}\label{kernel}
		Let  $\mathcal{I}_1(g, h, f) := \int g_*hf\theta^2b(\cos\theta)d\sigma dvdv_*, \mathcal{I}_2(g, h, f) := \int g_*hf'\theta^2b(\cos\theta)d\sigma dvdv_*$. If $s < 1/2$, $\mathcal{I}_3(g, h, f) := \int g_*hf'\theta b(\cos\theta)d\sigma dvdv_*.$ For $1\le k\le 3$, it holds that
		\begin{equation*}
			|\mathcal{I}_k(g, h, f)| \lesssim |g|_{L^1}|h|_{L^2}|f|_{L^2}.
		\end{equation*}
	\end{lemma}

	\begin{proof} We only give a detailed proof to $\mathcal{I}_2$ and $\mathcal{I}_3$. We first observe that
		\begin{equation}\label{kernel1}
			\int\theta^2b(\cos\theta)d\sigma \sim \int_0^{\pi/2}\theta^{1 - 2s}d\theta \sim 1.
		\end{equation}
	By Cauchy-Schwartz inequality and Lemma \ref{chv}, we have 
		\begin{equation*}
			|\mathcal{I}_2(g, h, f)| \lesssim \big(\int|g_*||h|^2dvdv_*\big)^{1/2} \times \big(\int|g_*||f|^2dvdv_*\big)^{1/2} \le |g|_{L^1}|h|_{L^2}|f|_{L^2}.
		\end{equation*}
  Note that if $s < 1/2$, then
		\begin{equation}\label{kernel2}
			\int\theta b(\cos\theta)d\sigma \sim \int_0^{\pi/2}\theta^{-2s}d\theta \sim 1.
		\end{equation}
	Again by Cauchy-Schwartz inequality and Lemma \ref{chv}, we can get the desired result. This ends
		 the proof.
	\end{proof}

	\begin{lemma}\label{singular} Let $\mathcal{J}_1(g, h, f) := \int g_*hf|v - v_*|^{\gamma + 2}\theta^2b(\cos\theta)1_{|v' - v| \le 1}d\sigma dvdv_*$, $\mathcal{J}_2(g, h, f) := \int g_*hf'|v - v_*|^{\gamma + 2}\theta^2b(\cos\theta)1_{|v' - v| \le 1}d\sigma dvdv_*$ and $\mathcal{J}_3(g, h, f) := \int g_*h'f'|v - v_*|^{\gamma + 2}\theta^2b(\cos\theta)1_{|v' - v| \le 1}d\sigma dvdv_*.$ For $1\le k\le 3$, it holds that
		\begin{equation*}
			|\mathcal{J}_k(g, h, f)| \lesssim |g|_{L_\mathfrak{r}^1}|h|_{L_{\omega_1}^2}|f|_{L_{\omega_2}^2}, \:\:\: \omega_1, \omega_2 \ge 0, \omega_1 + \omega_2 = \mathfrak{r}.
		\end{equation*}
	\end{lemma}

	\begin{proof} We separate the proof into three parts.
\smallskip

		\noindent \underline{Estimate of $\mathcal{J}_1(g, h, f)$.} If $\gamma + 2s > 0$, by the definition,  $\mathfrak{r} = \gamma + 2s$. Use $|v' - v| = \sin(\theta/2)|v - v_*|$, then
		\begin{equation}\label{singular1}
			\int\theta^2b(\cos\theta)1_{|v' - v| \le a}d\sigma \sim \int_0^{\pi/2}\theta^{1 - 2s}1_{\sin\frac{\theta}{2} \le a|v - v_*|^{-1}}d\theta \lesssim a^{2 - 2s}|v - v_*|^{2s - 2}.
		\end{equation}
		Then the desired result follows since
		\begin{equation*}
			|\mathcal{J}_1(g, h, f)| \lesssim \int|g_*hf||v - v_*|^{\gamma + 2s}dvdv_* \le \int|\langle v_* \rangle^rg_*\langle v \rangle^{\omega_1}h\langle v \rangle^{\omega_2}f|dvdv_* \le |g|_{L_r^1}|h|_{L_{\omega_1}^2}|f|_{L_{\omega_2}^2}.
		\end{equation*}

		If $\gamma + 2s \le 0, \mathfrak{r}  = 0$. We split integral region into two parts: $|v - v_*| > 1$ and $|v - v_*| \le 1$. For the first one, we   use \eqref{singular1} to get that
		\begin{equation*}
			|\mathcal{J}_1(g, h, f)| \lesssim \int|g_*hf||v - v_*|^{\gamma + 2s}dvdv_* \le \int|g_*hf|dvdv_* \le |g|_{L^1}|h|_{L^2}|f|_{L^2}.
		\end{equation*}
		For the second one, since $\gamma + 2 > 0$, by \eqref{kernel1}, we have
		\begin{equation*}
			|\mathcal{J}_1(g, h, f)| \lesssim \int|g_*hf|\theta^2b(\cos\theta)d\sigma dvdv_* \lesssim \int|g_*hf|dvdv_* \le |g|_{L^1}|h|_{L^2}|f|_{L^2}.
		\end{equation*}

		\noindent\underline{Estimate of $\mathcal{J}_2(g, h, f)$.} If $\gamma + 2s > 0, \mathfrak{r}= \gamma + 2s$. By Lemma \ref{v'bound}, we have $|v - v_*|^\mathfrak{r} \lesssim \langle v_* \rangle^\mathfrak{r}\langle v \rangle^{\omega_1}\langle v' \rangle^{\omega_2}$.  Cauchy-Schwartz inequality implies that
		\begin{equation*}
		\begin{aligned}
			|\mathcal{J}_2(g, h, f)|  
			&\le \int|\langle v_* \rangle^rg_*\langle v \rangle^{\omega_1}h\langle v' \rangle^{\omega_2}f'||v - v_*|^{2 - 2s}\theta^2b(\cos\theta)1_{|v' - v| \le 1}d\sigma dvdv_* \\
			&\le \big(\int|\langle v_* \rangle^rg_*||\langle v \rangle^{\omega_1}h|^2|v - v_*|^{2 - 2s}\theta^2b(\cos\theta)1_{|v' - v| \le 1}d\sigma dvdv_*\big)^{1/2} \\
			&\times \big(\int|\langle v_* \rangle^rg_*||\langle v' \rangle^{\omega_2}f'|^2|v - v_*|^{2 - 2s}\theta^2b(\cos\theta)1_{|v' - v| \le 1}d\sigma dvdv_*\big)^{1/2}.\\
		\end{aligned}
		\end{equation*}
		By Lemma \ref{chv} and \eqref{singular1}, we have
		\begin{equation*}
			|\mathcal{J}_2(g, h, f)| \lesssim \big(\int|\langle v_* \rangle^\mathfrak{r}g_*||\langle v \rangle^{\omega_1}h|^2dvdv_*\big)^{1/2} \times \big(\int|\langle v_* \rangle^\mathfrak{r}g_*||\langle v \rangle^{\omega_2}f|^2dvdv_*\big)^{1/2} \le |g|_{L_\mathfrak{r}^1}|h|_{L_{\omega_1}^2}|f|_{L_{\omega_2}^2}.
		\end{equation*}

		If $\gamma + 2s \le 0, \mathfrak{r} = 0$, again we split the integral region into: $|v - v_*| > 1$ and $|v - v_*| \le 1$. For the first one, using the fact that$|v - v_*|^{\gamma + 2} \lesssim |v - v_*|^{2 - 2s}$, we may use Cauchy-Schwartz inequality as above to get the desired result. While for the second one, we  transfer the estimate  to $\mathcal{I}_2(g, h, f)$ to get the desired result.

		\noindent\underline{Estimate of $\mathcal{J}_3(g, h, f)$.} By Lemma \ref{chv}, we directly have
		\begin{equation*}
			|\mathcal{J}_3(g, h, f)| \lesssim \int|g_*hf|\theta^2b(\cos\theta)1_{|v' - v| \le \cos^3(\theta/2)}d\sigma dvdv_*.
		\end{equation*}
		Then use \eqref{singular1} and Cauchy-Schwartz inequality to complete the proof.
	\end{proof}

	\begin{lemma}\label{nonsingular}
		Let $\mathcal{N}_1(g, h, f) := \int g_*hfB1_{|v' - v| > 1}d\sigma dvdv_*$, $\mathcal{N}_2(g, h, f) := \int g_*hf'B1_{|v' - v| > 1}d\sigma dvdv_*$ and $\mathcal{N}_3(g, h, f) := \int g_*h'f'B1_{|v' - v| > 1}d\sigma dvdv_*$. It holds that for $1\le k\le 3$, 
		\begin{equation*}
			|\mathcal{N}_k(g, h, f)| \lesssim |g|_{L_\mathfrak{r}^1}|h|_{L_{\omega_1}^2}|f|_{L_{\omega_2}^2}, \:\:\: \omega_1, \omega_2 \ge 0, \omega_1 + \omega_2 = \mathfrak{r}.
		\end{equation*}
	\end{lemma}

	\begin{proof} We only give a detailed proof for $\mathcal{N}_2(g, h, f)$. The others can be proven similarly. We first observe that
\begin{equation}\label{nonsingular1}
			\int b(\cos\theta)1_{|v' - v| > a}d\sigma \sim \int_0^{\pi/2}\theta^{-1 - 2s}1_{\sin\frac{\theta}{2} > a|v - v_*|^{-1}}d\theta \lesssim a^{-2s}|v - v_*|^{2s}.
		\end{equation}
	 Thanks to the definition of $\mathfrak{r}$,  by  Lemma \ref{v'bound}, we 
  always have
		\begin{equation*}
		\begin{aligned}
			|\mathcal{N}_2(g, h, f)| &\lesssim \int|\langle v_* \rangle^\mathfrak{r}g_*\langle v \rangle^{\omega_1}h\langle v' \rangle^{\omega_2}f'||v - v_*|^{-2s}b(\cos\theta)1_{|v' - v| > 1}d\sigma dvdv_*.\\
			&\le \big(\int|\langle v_* \rangle^\mathfrak{r}g_*||\langle v \rangle^{\omega_1}h|^2|v - v_*|^{-2s}b(\cos\theta)1_{|v' - v| > 1}d\sigma dvdv_*\big)^{1/2} \\
			&\times \big(\int|\langle v_* \rangle^\mathfrak{r}g_*||\langle v' \rangle^{\omega_2}f'|^2|v - v_*|^{-2s}b(\cos\theta)1_{|v' - v| > 1}d\sigma dvdv_*\big)^{1/2}.\\
		\end{aligned}
		\end{equation*}
		By Lemma \ref{chv} and \eqref{nonsingular1}, we have
		\begin{equation*}
			|\mathcal{N}_2(g, h, f)| \lesssim \big(\int|\langle v_* \rangle^\mathfrak{r}g_*||\langle v \rangle^{\omega_1}h|^2dvdv_*\big)^{1/2} \times \big(\int|\langle v_* \rangle^\mathfrak{r}g_*||\langle v \rangle^{\omega_2}f|^2dvdv_*\big)^{1/2} = |g|_{L_r^1}|h|_{L_{\omega_1}^2}|f|_{L_{\omega_2}^2}.
		\end{equation*} This ends the proof.
	\end{proof}

	\begin{lemma}\label{nonsingular+} Let $\mathcal{P}(g, h, f) := \int g_*hf'|v - v_*|^{-2s}b(\cos\theta)1_{|v' - v| > 1}d\sigma dvdv_*$. It holds that \[|\mathcal{P}(g, h, f)| \lesssim |g|_{L^1}|h|_{L^2}|f|_{L^2}.\]
			 
	\end{lemma}

	\begin{proof}
		By Cauchy-Schwartz inequality we have
		\begin{equation*}
		\begin{aligned}
			|\mathcal{P}(g, h, f)| &\le \big(\int|g_*||h|^2|v - v_*|^{-2s}b(\cos\theta)1_{|v' - v| > 1}d\sigma dvdv_*\big)^{1/2} \\
			&\times \big(\int|g_*||f'|^2|v - v_*|^{-2s}b(\cos\theta)1_{|v' - v| > 1}d\sigma dvdv_*\big)^{1/2}.\\
		\end{aligned}
		\end{equation*}
		Thanks to  Lemma \ref{chv} and \eqref{nonsingular1}, we have
		\begin{equation*}
			|\mathcal{P}(g, h, f)| \lesssim \big(\int|g_*||h|^2dvdv_*\big)^{1/2} \times \big(\int|g_*||f|^2dvdv_*\big)^{1/2} = |g|_{L^1}|h|_{L^2}|f|_{L^2}.
		\end{equation*}
		This ends the proof.
	\end{proof}

 {\bf Acknowledgments.} Ling-Bing He  and Wu-Wei Li are supported by NSF of China under Grants 12141102.

\end{document}